\documentclass[twoside,10pt]{amsart}
  \usepackage{amsmath,amsfonts,amssymb,amsthm,mathtools,multirow,tabls,mathrsfs}
  \usepackage[a4paper,margin=1.1in]{geometry}  
  \usepackage[all]{xy}
  \usepackage[utf8]{inputenc}
  \usepackage[T1]{fontenc}
  \usepackage[shortlabels]{enumitem}
  \usepackage[colorlinks,citecolor=magenta,linkcolor=black]{hyperref}

  \numberwithin{equation}{subsection}
  \newtheorem{thm}{Theorem}[subsection]
\newtheorem{lemma}[thm]{Lemma}
\newtheorem{prop}[thm]{Proposition}
\newtheorem{cor}[thm]{Corollary}

\newtheorem{thmi}{Theorem}     
\newtheorem{conji}{Conjecture} 

\theoremstyle{definition} 
\newtheorem{defin}[thm]{Definition}
\newtheorem{remark}[thm]{Remark}
\newtheorem{variant}[thm]{Variant}
\newtheorem{example}[thm]{Example}

\makeatletter
  \renewcommand\subsubsection{\@startsection{subsubsection}{2}%
         \z@{.5\linespacing\@plus.7\linespacing}{-.5em}%
    {\normalfont\scshape}}
\makeatother

\newcommand{\cD}{\mathscr D}
\newcommand{\cE}{\mathscr E}
\newcommand{\cF}{\mathscr F}
\newcommand{\cG}{\mathscr G}
\newcommand{\cH}{\mathscr H}

\newcommand{\cL}{\mathscr L}
\newcommand{\cM}{\mathscr M}
\newcommand{\cO}{\mathscr O}
\newcommand{\cR}{\mathscr R}

\newcommand{\cX}{\mathscr X}
\newcommand{\cY}{\mathscr Y}

\newcommand{\bb}[1]{\mathbf{#1}} 
\newcommand{\FF}{\bb{F}}
\newcommand{\GG}{\bb{G}}
\newcommand{\PP}{\bb{P}}
\newcommand{\QQ}{\bb{Q}}
\newcommand{\RR}{\bb{R}}
\newcommand{\ZZ}{\bb{Z}}

\renewcommand{\phi}{\varphi}
\newcommand{\isom}{\simeq} 
\newcommand{\piet}{\pi^{\acute{e}t}_{1}}
\newcommand{\et}{{\rm\acute{e}t}}
\newcommand{\cat}[1]{{\normalfont\textbf{#1}}}  
\newcommand{\ra}{\longrightarrow}    
\newcommand{\isomto}{\xrightarrow{\,\smash{\raisebox{-0.65ex}{\ensuremath{\scriptstyle\sim}}}\,}}
\renewcommand{\bigwedge}{\mbox{\large $\wedge$}}

\DeclareMathOperator{\Alb}{Alb}
\DeclareMathOperator{\Art}{\mathbf{Art}} 
\DeclareMathOperator{\Aut}{Aut}

\DeclareMathOperator{\Def}{Def}
\DeclareMathOperator{\Ext}{Ext}

\DeclareMathOperator{\Hom}{Hom}
\DeclareMathOperator{\Pic}{Pic}

\DeclareMathOperator{\Spec}{Spec}

\newcommand{\cExt}{{\mathscr E}\kern -.5pt xt}
\newcommand{\cHom}{\mathscr{H}\kern -.5pt om}
\newcommand{\cEnd}{{\mathscr E}nd}

\newcommand{\stacksproj}[1]{{\cite[Tag~\href{http://stacks.math.columbia.edu/tag/#1}{#1}]{stacks-project}}}

\newcommand{\wt}[1]{{\mathchoice%
  {\raisebox{1.5ex}{\resizebox{1.7ex}{!}{{}\hphantom{i}\ensuremath{{\sim}}}} \hspace{-1.7ex}{#1}}%
  {\smash{\raisebox{1.5ex}{\resizebox{1.7ex}{!}{{}\hphantom{i}\ensuremath{{\sim}}}}\hspace{-1.7ex}{#1 }}\vphantom{\tilde I}}%
  {\raisebox{1.1ex}{\resizebox{1.3ex}{!}{{}\hphantom{i}\ensuremath{{\sim}}}}\hspace{-1.3ex}{#1}}%
  {\raisebox{0.8ex}{\resizebox{1ex}{!}{{}\hphantom{i}\ensuremath{{\sim}}}}\hspace{-1ex}{#1}}%
}}

\newcommand{\swt}[1]{{\mathchoice%
  {\raisebox{0.9ex}{\resizebox{1.2ex}{!}{\ensuremath{{\sim}}}}\hspace{-1.4ex}{#1}}%
  {\smash{\raisebox{0.9ex}{\resizebox{1.2ex}{!}{\ensuremath{{\sim}}}}\hspace{-1.4ex}{#1 }}\vphantom{I}}%
  {\raisebox{0.7ex}{\resizebox{0.8ex}{!}{\ensuremath{{\sim}}}}\hspace{-0.9ex}{#1}}%
  {\raisebox{0.5ex}{\resizebox{1ex}{!}{{}\hphantom{i}\ensuremath{{\sim}}}}\hspace{-1ex}{#1}}%
}}

\newcommand{\wtcR}{\smash{\raisebox{1.5ex}{\hspace{0.7ex}\resizebox{1.2ex}{!}{\ensuremath{{\sim}}}}\hspace{-2.1ex}{\cR}}\vphantom{I}}

  \author[P.\ Achinger]{Piotr Achinger}
  \address{Institute of Mathematics, Polish Academy of Sciences
    \newline\indent ul. Śniadeckich 8, 00-656 Warsaw, Poland
  }
  \email{pachinger@impan.pl}

  \author[J.\ Witaszek]{Jakub Witaszek}
  \address{Department of Mathematics,University of Michigan, Ann Arbor, MI 48109} 
  \email{jakubw@umich.edu}

  \author[M.\ Zdanowicz]{Maciej Zdanowicz}
  \address{\'Ecole Polytechnique F\'ed\'erale de Lausanne, Chair of Algebraic Geometry \newline 
    \indent MA C3 585 (Bâtiment MA), Station 8, CH-1015 Lausanne}
  \email{maciej.zdanowicz@epfl.ch}
  
  \title[Global Frobenius Liftability I]{Global Frobenius Liftability I}
  \date{\today}

\begin{document}

\begin{abstract}
  We formulate a conjecture characterizing smooth projective varieties in positive characteristic whose Frobenius morphism can be lifted modulo $p^2$ --- we expect that such varieties, after a finite \'etale cover, admit a toric fibration over an ordinary abelian variety. We prove that this assertion implies a conjecture of Occhetta and Wiśniewski, which states that in characteristic zero a smooth image of a projective toric variety is a toric variety. To this end we analyse the behaviour of toric varieties in families showing some generization and specialization results. Furthermore, we prove a positive characteristic analogue of Winkelmann's theorem on varieties with trivial logarithmic tangent bundle (generalising a result of Mehta--Srinivas), and thus obtaining an important special case of our conjecture. Finally, using deformations of rational curves we verify our conjecture for homogeneous spaces, solving a problem posed by Buch--Thomsen--Lauritzen--Mehta.
\end{abstract}

\maketitle

\section{Introduction}
\label{s:intro}

\subsection{Liftings of Frobenius}
\label{ss:intro-flift}

One of the salient features of algebraic geometry in positive characteristic is the existence of the Frobenius morphism $F_X\colon X\to X$ for every $\FF_p$-scheme $X$, defined as the $p$-th power map $f\mapsto f^p$ on $\cO_X$. At a philosophical level, this paper together with its sequel \cite{PartII} argues that this type of structure becomes extremely rare as we move towards characteristic zero. This idea is not new: for example, in Borger's point of view on $\FF_1$-geometry \cite{Borger}, liftings of the Frobenius (compatible for all $p$) are seen as `descent data to $\FF_1$.' In more classical algebraic geometry, it is expected \cite{amerik97,fakhruddin} that projective varieties in characteristic zero admitting a~polarized endomorphism are very scarce. 

More precisely, we will be mostly concerned with the following question:

\medskip
\begin{center} 
  \emph{Which smooth projective varieties in characteristic $p$ lift modulo $p^2$ together with Frobenius?}
\end{center}
\medskip

For brevity, we will call such varieties \emph{$F$-liftable}. This question also has a long history: in their very influential paper \cite{DeligneIllusie}, Deligne and Illusie gave an algebraic proof of Kodaira--Akizuki--Nakano vanishing 
\[
  H^j(X, \Omega^i_X\otimes L) = 0\quad (L\text{ ample}, \quad i+j>\dim X)
\]
for smooth complex varieties by reducing the setting to characteristic $p>0$. The key ingredient was the analysis of \emph{local} liftings of Frobenius modulo $p^2$ and their relation to the de~Rham complex, namely that such a lifting $\wt F$ induces an injective homomorphism
\begin{equation} \label{eqn:xi-intro}
  \xi = \frac{d\wt F}{p} \colon F_X^* \Omega^1_{X} \ra \Omega^1_{X}
\end{equation}
whose adjoint $\Omega^1_{X}\to F_{X*} \Omega^1_{X}$ induces the inverse of the Cartier operator. In \cite{BTLM}, the authors study smooth projective varieties in characteristic $p>0$ admitting a \emph{global} lifting modulo $p^2$ together with Frobenius, and show in particular that such varieties satisfy a much stronger form of Kodaira vanishing called Bott vanishing
\begin{equation} \label{eqn:bott-intro}
  H^j(X, \Omega^i_X\otimes L) = 0\quad (L\text{ ample}, \quad j>0).
\end{equation}
This type of vanishing is extremely restrictive, and hence so is $F$-liftability.  In fact, all known examples of $F$-liftable varieties are in some sense built from toric varieties (where the lifting of Frobenius extends the multiplication by $p$ map on the torus) and ordinary abelian varieties (by the Serre--Tate theory). Note that the existence of an injective map $\xi$ as above implies that $(1-p)K_X$ is effective, so in particular $X$ has non-positive Kodaira dimension; in fact, the section $\det(\xi)\in H^0(X, \smash{\omega_X^{1-p}})$ corresponds to a Frobenius splitting of $X$.

The main goal of this paper is to provide evidence for the following conjectural answer to the above question, and to relate it to some open problems in characteristic zero.

\begin{conji} \label{conj:froblift}
  Let $X$ be a smooth projective variety over an algebraically closed field $k$ of characteristic $p>0$. If $X$ is $F$-liftable, then there exists a finite \'etale Galois cover $f\colon Y\to X$ such that the Albanese morphism of $Y$
  \[ 
    a_Y\colon  Y\ra {\rm Alb}(Y)
  \] 
  is a toric fibration. In particular, if $X$ is simply connected (for example, if $X$ is separably rationally connected), then $X$ is a toric variety.
\end{conji}

See \S\ref{ss:toric-var} for the definition of a toric fibration. Conjecture~\ref{conj:froblift} \emph{almost} characterizes $F$-liftable smooth projective varieties, see Remark~\ref{rmk:partial-converse} for a discussion of the converse.

\subsection{Special cases of Conjecture~\ref{conj:froblift}}
\label{ss:conj1-special-cases}

A few important special cases were already known for some time. In the paper \cite{MehtaSrinivas}, to which our work owes a great deal, Mehta and Srinivas prove (among other things) that if $X$ is $F$-liftable and the canonical bundle $\omega_{X}$ is numerically trivial, then $X$ admits a finite \'etale Galois cover by an ordinary abelian variety. In a different direction, the case of homogeneous spaces was considered in the aforementioned paper \cite{BTLM}, where it was shown that many rational homogeneous spaces are not \mbox{$F$-liftable}, and conjectured that the only $F$-liftable ones are products of projective spaces. Recently, the case of minimal surfaces was considered in \cite{xin16}.

Our first contribution confirms the aforementioned expectation of \cite{BTLM}.

\begin{thmi}[see Theorem~\ref{thm:homogeneous-spaces}] \label{thm:homog}
  Conjecture~\ref{conj:froblift} is true if $X$ is a homogeneous space. More precisely, if $X$ is a smooth projective $F$-liftable variety whose automorphism group acts transitively, then $X$ is isomorphic to a product of projective spaces and an ordinary abelian variety.
\end{thmi}

Our method of proof is based on the study of rational curves on $X$. In the crucial special case of rational homogeneous spaces of Picard rank one, we analyze the geometry of an \'etale covering of $X$ induced by the map \eqref{eqn:xi-intro}, and its restrictions to rational curves belonging to a carefully chosen covering family. In the final step, we apply Mori's characterization of the projective space. In fact, this method shows that Conjecture~\ref{conj:froblift} holds if $T_X$ is nef and $X$ is a~Fano variety of Picard rank one (see Proposition~\ref{prop:homog-rank-one}).

In the sequel \cite{PartII}, we also verify Conjecture~\ref{conj:froblift} in low dimensions.

\begin{thmi} \label{thm:specialcases}
  Conjecture~\ref{conj:froblift} is true in the following cases.
  \begin{enumerate}[(a)]
    \item If $\dim X\leq 2$ \cite[Section 3]{PartII}.
    \item If $X$ is a Fano threefold from the Mori--Mukai classification \cite[Section 4]{PartII}.
  \end{enumerate}
\end{thmi}

\subsection{Relation to other problems}
\label{ss:intro-other-problems}

Let us mention three problems (in arbitrary characteristic) to which Conjecture~\ref{conj:froblift} is related.

\subsubsection*{Images of toric varieties} Consider a problem of the following type: given a projective variety $Z$, determine all \emph{smooth} projective varieties $X$ for which there exists a surjective morphism
\[ 
  \phi\colon Z\ra X.
\]
One of the first applications of Mori theory was Lazarsfeld's solution \cite{Lazarsfeld} to a problem of Remmert and van de Ven \cite{RemmertVDV}: if $Z\isom \PP^n$, then $X\isom \PP^n$ (or $X$ is a point). Subsequently, more general problems of this kind were considered:  for example, if $Z$ is an abelian variety, then $X$ admits a finite \'etale cover by a product of an abelian variety and projective spaces \cite{Debarre,HwangMok01,DemaillyHwangPeternell}. In \cite{occhetta_wisniewski}, Occhetta and Wi\'sniewski proved that if $Z$ is a~toric variety and $X$ has Picard rank one, then $X\isom \PP^n$ and were led to pose the following conjecture:

\begin{conji} \label{conj:jarekw}

A smooth complex projective variety $X$ admitting a surjective map $\phi\colon Z\to X$ from a complete toric variety $Z$ is a toric variety.
\end{conji}

This project gained momentum upon our discovery of a link between Conjecture~\ref{conj:froblift} and Conjecture~\ref{conj:jarekw}.

\begin{thmi}[see Theorem~\ref{thm:c1impliesc2}] \label{thm:relate-conjectures} 
  Conjecture~\ref{conj:froblift} for simply connected varieties implies Conjecture~\ref{conj:jarekw} in characteristic zero. 
\end{thmi}

The key step in the proof uses the functoriality of obstruction classes to lifting Frobenius, an idea due to the third author \cite{zdanowicz}: after reducing a given surjective map $\phi\colon Z\to X$ modulo $p^2$, one can relate the obstruction classes to lifting Frobenius on $X$ and $Z$. Since $Z$ is a toric variety, its reduction modulo $p$ is $F$-liftable, and so is the reduction of $X$. By the assumed case of Conjecture~\ref{conj:froblift}, the reduction of $X$ is a toric variety. Finally, one uses a strengthening of Jaczewski's characterization of toric varieties \cite{Jaczewski} from \cite{KedzierskiWisniewski} to conclude that $X$ must be a toric variety. We obtain the following result as a by-product.

\begin{thmi}
  Let $\phi\colon Z\to X$ be a surjective morphism from a complete toric variety $Z$ to a smooth projective variety $X$ defined over a field $k$ of characteristic zero. Then $X$ satisfies Bott vanishing \eqref{eqn:bott-intro}.
\end{thmi}

Given the theorem it is natural to ask whether all smooth rationally connected varieties satisfying Bott vanishing are in fact toric.  After this article was made available as a preprint, Burt Totaro  answered this question negatively \cite{totaro_bott_vanishing} in a beautiful manner, by proving that the non-toric surface obtained by blowing up $\mathbf{P}^2$ in four points in general position satisfies Bott vanishing.

\subsubsection*{Varieties with trivial log tangent bundle}

Let $X$ be a smooth projective variety over an algebraically closed field $k$, and let $D\subseteq X$ be a divisor with normal crossings. Suppose that the sheaf $\Omega^1_X(\log D)$ of differentials with log poles along $D$ is free. If $k$ has characteristic zero, then by a result of Winkelmann \cite{Winkelmann}, $X$ admits an action of a semi-abelian variety (an extension of an abelian variety by a torus) which is transitive on $X\setminus D$; in particular, if $D=0$, then $X$ is an abelian variety. However, in positive characteristic $X$ might not be an abelian variety itself  when $D=0$, but  then the main result of \cite{MehtaSrinivas} states that if $X$ is ordinary, then it admits an abelian variety as a finite \'etale Galois cover. It is not known if the ordinarity assumption is necessary.

The link with Frobenius liftability was first observed in \cite{MehtaSrinivas} in relation to the above result: if $X$ has a trivial tangent bundle and is ordinary, then it is $F$-liftable. Conversely, if $\omega_X$ is numerically trivial and $X$ is $F$-liftable, then the map \eqref{eqn:xi-intro} is an isomorphism, and it follows that $\Omega^1_X$ becomes trivial on a finite \'etale cover of $X$. 

We shall say that the pair $(X, D)$ is \emph{$F$-liftable} if there exists a lifting $(\wt X, \wt D)$ of $(X, D)$ modulo $p^2$ such that $\wt D$ has relative normal crossings, and a lifting $\wt F$ of Frobenius to $\wt{X}$ which is compatible with $\wt D$ in the sense that
\[ 
  \wt F{}^* \wt D = p\wt D.
\]

\begin{thmi}[see Theorem~\ref{thm:frobenius_cy}] \label{thm:log-mehta-srinivas} 
  Let $(X, D)$ be a projective nc pair over an algebraically closed field $k$ of positive characteristic. The following conditions are equivalent.
  \begin{enumerate}[(i)]
    \item $(X, D)$ is $F$-liftable and $\omega_X(D)$ is numerically trivial,
    \item $X$ is $F$-split and $\Omega^1_X(\log D)$ becomes trivial on a finite \'etale cover of $X$,
    \item $X$ admits a finite \'etale cover $f \colon Y \to X$ whose Albanese map $a\colon Y\to {\rm Alb}(Y)$ is a~toric fibration over an ordinary abelian variety with toric boundary $f^{-1}(D)$.
  \end{enumerate}
\end{thmi}

\noindent This characterization of varieties with trivial logarithmic tangent bundle is a positive characteristic analogue of Winkelmann's theorem \cite{Winkelmann} and a generalisation of the result of Mehta-Srinivas \cite{MehtaSrinivas}. Moreover, Theorem~\ref{thm:log-mehta-srinivas} shows that Conjecture~\ref{conj:froblift} is true when one can find a Frobenius compatible nc divisor $D$ such that $\omega_X(D)$ is numerically trivial.

The proof is not less complicated than the one in \cite{MehtaSrinivas}. We closely follow their strategy with a few differences. The case $H^1(X, \cO_X) = 0$ is handled by lifting to characteristic zero, applying the aforementioned theorem of Winkelmann \cite{Winkelmann}, and reducing the setting back to positive characteristic. Interestingly, our results show that a direct generalisation of Winkelmann's theorem to characteristic $p$ is false, while Winkelmann's theorem indicates that a natural analogue of Theorem~\ref{s:log-mehta-srinivas}(iii)$\Rightarrow$(ii) is false in characteristic zero (see Remark~\ref{rmk:false-theorems}). But the most important difference comes from the fact that \cite{MehtaSrinivas} used Yau's work on the Calabi conjecture for a~lifting of $X$ to characteristic zero in one of the key steps. In the log setting, such results are unfortunately unavailable, and instead we reduce to the case of a finite ground field, then lift to characteristic zero together with Frobenius, and use results of \cite{nakayama-zhang} and \cite{greb_kebekus_peternell} on varieties with a polarized endomorphism. This brings us to the next array of open problems.

\subsubsection*{Polarized endomorphisms} 

A polarized endomorphism of a projective variety $X$ is by definition one which extends to an endomorphism of an ambient projective space. It seems that there are many similarities between Frobenius liftings and polarized endomorphisms. Over a finite field, a power of the Frobenius morphism is a polarized endomorphism. In characteristic zero, the toric Frobenius is an example of such, as is the Serre--Tate canonical lifting of Frobenius on an ordinary abelian variety. As in the case of Frobenius liftability, varieties admitting a~polarized endomorphism satisfy $\kappa(X)\leq 0$ (see \cite{nakayama-zhang}).

If $\rho(X)=1$, then every endomorphism is polarized. In this case, a conjecture by Amerik \cite{amerik97} states that if a smooth rationally connected $X$ with $\rho(X)=1$ admits a polarized endomorphism, then $X \isom \PP^n$. This has been proved in certain cases, for example, when $\dim X \leq 3$ \cite{ARV}, $X$ is a hypersurface \cite{beauville}, or $X$ is a rational homogeneous space \cite{paranjape-Srnivas}. Moreover, Zhang showed that a smooth Fano threefold with a polarized endomorphism is rational \cite{zhang10}. In general, it has been asked if every smooth rationally connected variety with a polarized endomorphism is toric (see, for instance, \cite{fakhruddin}).

\subsection{Methods and further results} 

The proofs of the results mentioned above rely on a variety of tools: deformation theory, toric fibrations and the study of Frobenius liftings and their consequences.  Let us mention here some of our observations which might be of independent interest. Some of these results are valid for arbitrary (non-necessarily smooth) $k$-schemes.  

Two of the most important tools related to Frobenius liftability are the technique of descent (see Theorem~\ref{thm:descending-frob-lift}) and the sheaf of invariants differentials (see \S\ref{ss:def-thy-frob}).  The descent result states roughly that if $Y$ is $F$-liftable and $\pi \colon Y\to X$ is a morphism, then under certain assumptions on $\pi$, one can deduce the $F$-liftability of $X$.  We provide a more comprehensive treatment of the subject in the sequel \cite{PartII}.

We also provide a few results on toric varieties in families.  First, if a~general fiber of a~smooth projective family is toric, then the generic fiber must be toric as well (Corollary~\ref{cor:toric-generalization}).  To this end, we need an auxiliary result (Proposition~\ref{prop:split}) on the constructibility of the locus where a given vector bundle splits. Second, given a smooth projective family together with a relative normal crossings divisor over a connected base, if a \emph{single} geometric fiber is a toric pair, then the entire family has to be a toric fibration (Proposition~\ref{prop:global-rigidity}). In other words, using the terminology of Hwang and Mok, toric pairs are globally rigid.

\subsection{Outline of the paper}
\label{ss:outline}

In Section~\ref{s:prelim}, we gather some preliminary results and notions: toric varieties over an arbitrary base (i.e., toric fibrations) in \S\ref{ss:toric-var}, deformation theory including the crucial technique of descending deformations in \S\ref{ss:def-thy}, some basic facts about normal crossing pairs in \S\ref{ss:nc-pairs}, a review of the Cartier operator in \S\ref{ss:log_cartier_isom}, and Frobenius splittings in \S\ref{ss:f-split}. Section~\ref{s:froblift} deals with Frobenius liftability, mainly of smooth schemes and from the point of view of deformation theory (see \cite[\S 2]{PartII} for more general results on Frobenius liftability). In Section~\ref{s:toric-reductions}, we deal with the abstract problem of showing that the generic fiber of a family is toric if a~general one is (\S\ref{ss:generalization}) and with the opposite problem of proving that if the generic fiber of a~degeneration is a toric pair, then so is the special one (\S\ref{ss:specialization}). These results are then applied to show a general `global rigidity' result for toric pairs (\S\ref{ss:global-rigidity}) and Theorem~\ref{thm:relate-conjectures} (\S\ref{ss:images}). Sections~\ref{s:log-mehta-srinivas} and \ref{s:homogeneous-spaces} deal with Theorems~\ref{thm:log-mehta-srinivas} and \ref{thm:homog}, respectively. 

\subsection{Notation and conventions}
\label{ss:notation-conventions}

Let $p$ be a prime number. If $X$ is an $\FF_p$-scheme, we denote by $F_X\colon X\to X$ its absolute Frobenius morphism (i.e., the identity on the underlying topological space and the $p$-th power map on the structure sheaf). If $f \colon X \to S$ is a morphism of $\FF_p$-schemes, the relative Frobenius $F_{X/S}\colon X\to X'$ is the unique morphism making the following diagram commute
\begin{equation} \label{eqn:frobdgm} 
  \xymatrix{
    X \ar@/^1em/[drr]^{F_X} \ar@/_1em/[ddr]_f  \ar@{.>}[dr]|{F_{X/S}} \\
    & X' \ar[d] \ar[r]^{W_{X/S}} \ar@{}[dr]|\square & X\ar[d]^f \\
    & S \ar[r]_{F_S} & S.
  }
\end{equation}
The $S$-scheme $X'$ is called the \emph{Frobenius twist} of $X$ relative to $S$. Note that if $S=\Spec k$ for a~perfect field $k$, then $W_{X/S}$ is an isomorphism of schemes.

Most of the time we shall be working with a fixed perfect field $k$ of characteristic ${p>0}$. We denote by $W_n(k)$ its ring of Witt vectors of length $n$, that is, the unique (up to a~unique isomorphism) flat $\ZZ/p^n\ZZ$-algebra with an isomorphism $W_n(k)/pW_n(k)\isom k$, and by $W(k) = \smash{\varprojlim_n} W_n(k)$ the full ring of Witt vectors. The unique endomorphism restricting to the Frobenius $F_k$ modulo $p$ is denoted by $\sigma\colon W_n(k)\to W_n(k)$ or $\sigma\colon W(k)\to W(k)$. 

We often set $S=\Spec k$ and $\wt S = \Spec W_2(k)$. If $X$ is an $S$-scheme, a \emph{lifting of $X$ (modulo~$p^2$)} is a~flat $\wt S$-scheme $\wt X$ with an isomorphism $\smash{\wt X\times_{\wt S} S} \isom X$. A \emph{lifting of Frobenius} on $X$ to $\wt X$ is a morphism $\wt F_X \colon \wt X\to \wt X$ restricting to $F_X$ modulo $p$. We shall give a thorough discussion of these in~\S\ref{s:froblift}.

\subsection*{Acknowledgements}
We would like to thank Renee Bell, Paolo Cascini, Nathan Ilten, Adrian Langer, Arthur Ogus, Vasudevan Srinivas, and Jarosław Wiśniewski for helpful suggestions and comments. 

Part of this work was conducted during the miniPAGES Simons Semester at the Banach Center in Spring 2016.  The authors would like to thank the Banach Center for hospitality. The first author was supported by NCN OPUS grant number UMO-2015/17/B/ST1/02634. The second author was supported by the Engineering and Physical Sciences Research Council [EP/L015234/1].  The third author was supported by NCN PRELUDIUM grant number UMO-2014/13/N/ST1/02673.  This work was partially supported by the grant 346300 for IMPAN from the Simons Foundation and the matching 2015--2019 Polish MNiSW fund. 


\section{Preliminaries}
\label{s:prelim}

\subsection{Toric varieties and toric fibrations}
\label{ss:toric-var}

A toric variety (see e.g.\ \cite{Fulton}) is by definition a normal algebraic variety $X$ over an algebraically closed field $k$ together with an effective action of a torus $T\isom \GG_m^n$ with a dense orbit. Such varieties admit a completely combinatorial description in terms of \emph{rational polyhedral fans} $\Sigma$ in $N_{\RR} = \Hom(\GG_m, T)\otimes_{\ZZ} \RR$. This description is independent of the field $k$, and in particular every toric variety has a natural model $X(\Sigma)$ over $\ZZ$. We denote by $D(\Sigma)\subseteq X(\Sigma)$ the \emph{toric boundary} of $X(\Sigma)$, that is, the complement of the open orbit. Sometimes we shall abuse the terminology and say that a~variety $X$ \emph{is} a toric variety meaning that it admits the structure of a~toric variety. 

In what follows, we will have to deal with families of toric varieties over more general bases (such as the rings $W_2(k)$ and $W(k)$ or the Albanese variety in the statement of Conjecture~\ref{conj:froblift}). There is more than one sensible definition of a `toric variety over a base scheme $S$,' and we decided to settle on the following. 

\begin{defin} 
  Let $S$ be a scheme. 
  \begin{enumerate}[(a)]
    \item A \emph{torus} over $S$ is an $S$-group scheme $T$ which is \'etale-locally isomorphic to $\GG_{m,S}^n$ for some $n\geq 0$.
    \item A \emph{toric fibration} over $S$ is a flat $S$-scheme $X$ together with an action of a torus $T$ over $S$ such that \'etale-locally on $S$ there exist isomorphisms $T\isom \GG_{m,S}^n$ and $X\isom X(\Sigma)_S$ for some rational polyhedral fan $\Sigma\subseteq \RR^n$.
  \end{enumerate}
\end{defin}

In particular, we do not require the torus $T$ to be \emph{split}, i.e.\ that $T\isom \GG^n_{m, S}$. In general, $\Aut_S(\GG^n_{m,S})\isom GL(n, \ZZ)_S$, and hence tori of dimension $n$ over $S$ are parametrized by the \'etale non-abelian cohomology $H^1(S, GL(n, \ZZ))$; if $S$ is connected and normal, this is $\Hom(\pi_1(S, \bar s), GL(n, \ZZ))$. Since $\pi_1(S, \bar s)$ is pro-finite and $GL(n, \ZZ)$ is discrete, every such homomorphism has finite image, and we see that every torus over a normal $S$ becomes split on a finite \'etale cover of $S$. If $T$ is split, then we will say that $X$ is a \emph{split} toric fibration. If $X\to S$ is a toric fibration for a torus $T$, the (relative) dense open orbit $U \subseteq X$, i.e.\ the biggest open on which $T$ acts freely, is a $T$-torsor over $S$. \'Etale-locally on $S$, we have $T\isom \GG_{m,S}^n$ and $X\isom X(\Sigma)_S$ as in (b) above, and then $U \isom (X(\Sigma)\setminus D(\Sigma))_S$. If this torsor is trivial (i.e., $U\to S$ admits a section), we shall say that $X$ is a \emph{trivial} toric fibration. The usual description of toric varieties over algebraically closed fields generalizes as follows:

\begin{lemma}
  Suppose that $S$ is connected and let $X\to S$ be a split toric fibration under a torus $T$. There exists a fan $\Sigma$ in $N_\RR$ and a $T$-torsor $U\to S$ such that
  \[ 
  X \isom U \times^T X(\Sigma)_S
  \]
  as $S$-schemes with a $T$-action. In particular, if $X$ is moreover trivial, then $X\isom X(\Sigma)_S$. 
\end{lemma}

\begin{proof}
Suppose first that $X$ is trivial. In this case, for a fixed fan $\Sigma$, the $T$-equivariant isomorphisms $X\isom X(\Sigma)_S$ form a $T$-torsor over $S$ which is easily seen to be trivial. A section of this torsor gives the required isomorphism. 

The general case can be reduced to the trivial case by pulling back $X\to S$ along $U\to S$: we obtain an isomorphism $X_U \isom X(\Sigma)_U$ and take quotients by $T$ on both sides to obtain $X = X_U/T \isom X(\Sigma)_U/T = U\times^T X(\Sigma)_S$.
\end{proof}

For a general toric fibration $X\to S$, we can apply the above lemma \'etale locally on $S$. The toric boundaries $D(\Sigma)_S$ glue together to give a global \emph{toric boundary} $D$, which is a closed subscheme supported on $X\setminus U$.

If $X$ is smooth over $S$, then $D$ has normal crossings relative to $S$, and $\smash{\Omega^1_{X/S}(\log D)}$ is trivial if moreover $T$ is split. (Indeed, a local section $v\in {\rm Lie}(T) = e^* T_{T/S}$ of the Lie algebra defines an $S$-map $v\colon S[\varepsilon] = \Spec \cO_S[\varepsilon]/(\varepsilon^2) \to T$, and the composition of $v\times {\rm id} \colon S[\varepsilon]\times_S X\to T\times_S X$ with the $T$-action gives a section of the trivial thickening $X\to X[\varepsilon] = S[\varepsilon]\times_S X$ over $S$, i.e.\ a section of $T_{X/S}$. This lies in the dual of $\Omega^1_{X/S}(\log D)$, and we obtain an isomorphism $\pi^* {\rm Lie}(T)^\vee \isomto \Omega^1_{X/S}(\log D)$.)  We call a pair $(X, D)$ of a smooth $S$-scheme $X$ and a divisor with relative normal crossings $D\subseteq X$ a \emph{toric pair} if at arises via the above construction. In \S\ref{ss:global-rigidity}, we will show that the torus $T$ and its action on $X$ are (essentially) uniquely defined by the pair $(X,D)$, at least if $X$ is projective over $S$. 

\begin{example} 
\begin{enumerate}[(1)]
  \item The projectivization $\PP_S(E)\to S$ of a vector bundle $E$ on $S$ has toric fibers, but it admits the structure of a split toric fibration only when $E$ is a~direct sum of line bundles.
  \item Let $S= \bb{A}^1_k$ and $X=\bb{A}^2_k \setminus \{(0,0)\}$, treated as an $S$-scheme via the first projection. Then $X\to S$ has a natural action of $\GG_{m, S}$, its fibers are toric varieties, but it is not a toric fibration.
  \item Let $S=\Spec \QQ$ and let $T=\{x^2 + y^2 = 1\}\subseteq \bb{A}^2_\QQ$ be the circle group. Let $X= \{x^2+y^2  = z^2\}\subseteq \PP^2_\QQ$ be the projective closure of $T$. The action of $T$ on itself extends to $X$, and $D=X\setminus T$ is a point of degree two. The map $X\to S$ is a toric fibration under $T$ which becomes split and trivial after base change to $\QQ(\sqrt{-1})$.
\end{enumerate}
\end{example}

\subsection{Deformation theory}
\label{ss:def-thy}

We use freely the standard results of deformation theory, for which we refer to \cite{hartshorne_deformation,illusie_cotangent}. If $k$ is a perfect field of characteristic $p>0$, we denote by $\Art_{W(k)}(k)$ the category of Artinian local $W(k)$-algebras with residue field $k$. If $X$ is a~$k$-scheme, we denote by
\[ 
  \Def_X \colon \Art_{W(k)}(k) \ra \cat{Set}, \quad
  A \mapsto \genfrac{\{}{\}}{0pt}{} {\text{isom.\ classes of flat }\wt X /\Spec A} {\text{with an identification }\wt X\otimes_A k \isom X}
\]
its deformation functor. Similarly, if $Z\subseteq X$ is a closed subscheme, we denote by $\Def_{X,Z}$ the functor of deformations $\wt X$ of $X$ together with an embedded deformation $\wt Z\subseteq \wt X$ of $Z$.  The following lemma provides a basic tool for descending liftings along fibrations.

\begin{lemma}[{\cite[Proposition 2.1]{liedtke_satriano}}] \label{lem:liedtke_satriano}
  Let $\pi \colon Y \to X$ be a morphism of $k$-schemes such that $\cO_X\isomto \pi_*\cO_Y$ and $R^1\pi_*\cO_Y = 0$.  Then there exists a natural transformation of deformation functors 
  \[ 
    \pi_* \colon \Def_{Y} \ra \Def_{X}, 
  \]
  associating to every lifting $\wt Y \in \Def_{Y}(A)$ a lifting $\wt X = \pi_*(\wt Y) \in \Def_{X}(A)$ together with an $A$-morphism $\swt \pi \colon \wt Y \to \wt X$ lifting $\pi$.  More precisely, the structure sheaf of $\wt X$ is defined by the formula $\smash{\cO_{\wt X} = \pi_*\cO_{\wt Y}}$, and consequently the pair $(\wt X, \swt\pi)$ is unique up to a unique isomorphism inducing the identity on $X$ and $\wt Y$. 
\end{lemma}

\subsection{Normal crossing pairs}
\label{ss:nc-pairs}

 We recall some basics on normal crossings (nc) pairs over general base schemes.

\begin{defin} \label{def:nc-pair}
  Let $S$ be a scheme. An \emph{nc pair} over $S$ is a pair $(X, D)$ of a smooth scheme $X$ over $S$ and a divisor $D\subseteq X$ with normal crossings relative to $S$ (see \cite[Exp.\ XIII \S 2.1]{SGA1}, \cite[\S 4]{katz_nilpotent}), that is, such that \'etale-locally on $X$ there exists an \'etale map $h\colon X\to\bb{A}^n_S$ with $D = h^*(\{x_1 \cdot \ldots \cdot x_n = 0\})$.
\end{defin}

\begin{remark} \label{rmk:nc-pair-vs-log-scheme}
It is often more natural and convenient to work in the framework of log geometry than with nc pairs. We decided not to do so, as we will need only a small subset of the theory which is easily handled by the classical results on nc pairs. A reader familiar with the language of log geometry will notice that everything we do with nc pairs can easily be phrased using log schemes instead, and that many of our results have natural logarithmic analogues. More precisely, the dictionary is as follows.

An \emph{nc log scheme} over a scheme $S$ is an fs log scheme $(X, \cM_X)$ over $(S, \cO_S^*)$ which is log smooth and such that the underlying map of schemes $X\to S$ is smooth. Equivalently, an fs log scheme $(X, \cM_X)$ over $(S, \cO_S^*)$ is nc if and only if \'etale locally on $X$, there exists a strict $h \colon (X, \cM_X)\to \bb{A}^n_S = \Spec(\bb{N}^n \to \cO_S[x_1, \ldots, x_n])$. Thus, there is a natural equivalence of groupoids, compatible with base change
  \[ 
    \left(\text{nc log schemes}/(S, \cO_S^*)\right) \isomto \left(\text{nc pairs}/S\right).
  \]
\end{remark} 

\medskip

Given an nc pair $(X, D)$ over $S$, one defines the module of differentials $\smash{\Omega^1_{X/S}(\log D)}$ (see \cite[\S 4]{katz_nilpotent}; equivalently, this is the dual of the submodule of $T_{X/S}$ consisting of derivations preserving the ideal of $D$). It is locally free and its determinant is isomorphic to $\omega_{X/S}(D)$. Classical deformation theory for smooth schemes has a natural analogue for nc pairs, and in particular is controlled by the groups $\Ext^i(\Omega^1_{X/S}(\log D), \cO_X)$ for $i=0,1,2$, see \cite{fumihiru_kato_log_smooth}. Moreover, if $D$ is the union of smooth divisors $D_1, \ldots, D_r$, there is a short exact sequence
\[ 
  0\ra \Omega^1_{X/S} \ra \Omega^1_{X/S}(\log D)\ra \bigoplus_{i=1}^r \cO_{D_i}\ra 0.
\]

\begin{lemma}\label{lemma:nc-pair-over-s-and-over-y}
  Let $(X, D)$ be an nc pair over a scheme $S$, let $Y$ be a smooth scheme over $S$, and let $f \colon X \to Y$ be a morphism over $S$. Then $(X, D)$ is an nc pair over $Y$ if and only if the morphism 
  \[ 
    f^* \Omega^1_{Y/S}\ra \Omega^1_{X/S}\ra \Omega^1_{X/S}(\log D)
  \]
  is injective and its cokernel is locally free. Further, in this case this cokernel equals $\Omega^1_{X/Y}(\log D)$.
\end{lemma}  

\begin{proof}
If $(X,D)$ is an nc pair over $Y$, then the assertions are clearly satisfied.  For the proof in the other direction, we use the language of log geometry.  Accordingly, we treat $(X,D) \to Y$ as a~morphism of log schemes.  The sheaf $\smash{\Omega^1_{X/Y}(\log D)}$ is isomorphic to the sheaf of relative log differentials of $(X,D) \to Y$ and therefore, since it is locally free, we may apply \cite[Proposition 3.12]{kazuya_kato_logarithmic} to see that the morphism $(X,D) \to Y$ is log smooth.  The scheme $X$ is smooth over $Y$ and hence $(X,D) \to Y$ is an nc log scheme over $Y$.  We conclude using the equivalence between nc log schemes and nc pairs (Remark~\ref{rmk:nc-pair-vs-log-scheme}).
\end{proof}

\subsection{The Cartier isomorphism} 
\label{ss:log_cartier_isom}

Here, we present basic properties of the Cartier isomorphism and its logarithmic variant.  Let $X \to S$ be a~smooth morphism of schemes over $k$, and let $\smash{\Omega^\bullet_{X/S}}$ be its de Rham complex.  Moreover, let $\smash{B^i_{X/S}}$ (resp.\ $\smash{Z^i_{X/S}}$) be the $i^{\rm th}$ coboundaries (resp.\ cocycles) in the $\cO_{X'}$-linear complex $\smash{F_{X/S*}\Omega^\bullet_{X/S}}$, where $X'$ is the Frobenius twist of $X$ relative to $S$. 

More generally, if $(X,D)$ in an nc pair over $S$, then we denote by $\smash{B^i_{X/S}(\log D)}$ (resp.\ $\smash{Z^i_{X/S}(\log D)}$) the $i^{\rm th}$ coboundaries (resp.\ cocycles) in the $\cO_{X^{'}}$-linear log de Rham complex $\smash{F_{X/S*}\Omega^\bullet_{X/S}(\log D)}$.  In analogy with the holomorphic Poincar\'e lemma, the following result describes the cohomology of the de Rham complex in characteristic $p$. 

\begin{thm}[{\cite[Theorem 7.2]{katz_nilpotent}}]\label{lemma:log_cartier}
  Let $X \to S$ be a smooth morphism of schemes over $k$.  Then there exists a unique system of isomorphisms of $\cO_{X'}$-modules
  \[
    C^{-1}_{X/S} \colon \Omega^j_{X'/S} \isomto \cH^j(F_{X/S*}\Omega^\bullet_{X/S})
  \]
  satisfying the conditions
  \begin{enumerate}[(i)]
    \item $C^{-1}_{X/S}(1) = 1$,
    \item $C^{-1}_{X/S}(\omega) \wedge C^{-1}_{X/S}(\eta) = C^{-1}_{X/S}(\omega \wedge \eta)$ for local sections $\omega \in \Omega^i_{X'/S}$ and $\eta \in \Omega^j_{X'/S}$,
    \item $C^{-1}_{X/S} \colon \Omega^1_{X'/S} \to \cH^1(F_{X/S*}\Omega^\bullet_{X/S})$ is defined by $d(g \otimes 1) \mapsto [g^{p-1}dg]$.
  \end{enumerate}
  The inverse isomorphisms give rise to short exact sequences
  \[
    0 \ra B^j_{X/S} \ra Z^j_{X/S} \ra \Omega^j_{X'/S} \ra 0,
  \]
  inducing \emph{Cartier morphisms} $C_{X/S} \colon Z^j_{X/S} \to \Omega^j_{X'/S}$.
\end{thm}

\begin{variant}
In the logarithmic setting, there exists a system of isomorphisms
\[
\Omega^j_{X'/S}(\log D') \isomto \cH^j(F_{X/S*}\Omega^\bullet_{X/S}(\log D)),
\] 
where $D'$ is the preimage of $D$ in $X'$.  Moreover, we have short exact sequences
  \[
    0 \ra B^j_{X/S}(\log D) \ra Z^j_{X/S}(\log D) \ra \Omega^j_{X'/S}(\log D') \ra 0,
  \]
inducing \emph{logarithmic Cartier morphisms} $C_{X/S} \colon Z^j_{X/S}(\log D) \to \Omega^j_{X'/S}(\log D')$.
\end{variant}

\subsection{Frobenius splittings}
\label{ss:f-split}

The standard reference for general facts about Frobenius splittings is \cite[Chapter I]{BrionKumar}. Let $k$ be a perfect field of characteristic $p>0$.

\begin{defin}
Let $X$ be a $k$-scheme.  A \emph{Frobenius splitting} on $X$ is an $\cO_X$-linear splitting $\sigma\colon F_{X*}\cO_X \to \cO_X$ of the map $F_X^*\colon \cO_X\to F_{X*} \cO_X$.  We say that a scheme is $F$-split if it admits a~Frobenius splitting.
\end{defin}

\noindent We shall need the following lemma describing basic properties of $F$-split schemes. 

\begin{lemma}\label{lem:f-split_is_ordinary}
  Let $X$ be a proper $F$-split scheme over an algebraically closed field $k$.  Then the following hold.
  \begin{enumerate}[(a)]
     \item The map $F^*_X \colon H^i(X,\cO_X) \to H^i(X,\cO_X)$ is bijective for all $i\geq 0$.
     \item The cohomology groups $H^i(X,B^1_X)$ vanish for all $i\geq 0$, where $B^1_X = F_{X*}\cO_X / \cO_X$.
     \item If $X$ is smooth, then the Albanese variety $\Alb X$ is ordinary.
     \item Every \'etale scheme $Y$ over $X$ is $F$-split.
   \end{enumerate} 
\end{lemma}

\begin{proof}
For (a) and (b), we use the long exact sequence of cohomology associated with
\[
  0 \ra \cO_X \ra F_{X*}\cO_X \ra B^1_X \ra 0,
\] 
and the fact that an injective $p^{-1}$-linear endomorphism of a finite dimensional vector space over a~perfect field is bijective. 

To prove (c), we reason as follows. By \cite[Lemma 1.3]{MehtaSrinivas} we see that $H^1(\Alb X,\cO_{\Alb X})$ injects into $H^1(X,\cO_X)$.  By (a) the Frobenius action on the latter group is bijective, and hence it is also bijective on the former.  This implies that $\Alb X$ is ordinary. 

For (d), let $f\colon Y\to X$ be an \'etale morphism. Then the square
\[
  \xymatrix{
    Y \ar[r]^{F_{Y}}\ar[d]_{f} & Y \ar[d]^{f} \\
    X \ar[r]_{F_X} & X
  }
\]
is cartesian by \cite[XIV=XV \S{}1 $n^\circ$2, Pr. 2(c)]{SGA5}, and hence $f^* F_{X*}\cO_X \isom F_{Y*} \cO_Y$ by flat base change. Thus given a splitting $\sigma \colon F_{X*}\cO_X \to \cO_X$, applying $f^*$ we obtain a morphism 
\[
  f^*(\sigma) \colon F_{Y*}\cO_{Y} \isom f^* F_{X*}\cO_X \ra \cO_{Y}
\]
which is a Frobenius splitting on $Y$.
\end{proof}


\section{Frobenius liftability}\label{s:frob-lift}
\label{s:froblift}

Throughout this section we fix a perfect field $k$ of characteristic $p>0$ (sometimes assumed algebraically closed).  First, set $S = \Spec k$ and $\wt S = \Spec W_2(k)$ (see \S\ref{ss:notation-conventions}). If $X$ is a~scheme over $S$ and $\wt X$ is a lifting of $X$ to $\wt S$, a \emph{lifting of Frobenius} on $X$ to $\wt X$ is a~morphism ${\wt F_X \colon  \wt X\to \wt X}$ restricting to $F_X$ on $X$. It automatically commutes with the map $\sigma \colon \wt S\to \wt S$ induced by the Witt vector Frobenius on $W_2(k)$ (see {\cite[Lemma 6.5.13 i)]{GabberRamero}}), and hence it fits into the commutative diagram
\[ 
    \xymatrix@R=1.5em@C=1.5em{
          & X\ar[rr] \ar[dd]|\hole & & \wt X \ar[dd] \\
        X \ar[ru]^{F_{X}}\ar[rr]\ar[dd] & & \wt X \ar[dd] \ar@{.>}[ru]^{\wt F_{X}} & \\
          & S \ar[rr]|(0.5)\hole & & \wt S \\
        S \ar[ru]^{F_S} \ar[rr] & & \wt S \ar[ru]_{\sigma}.
    }
\]
We call a pair $(\wt X, \wt F_X)$, where $\wt X$ and $\wt F_X$ are as above, a \emph{Frobenius lifting of $X$}. If such a~pair exists, we say that $X$ is \emph{$F$-liftable}. In this situation, we can form a diagram lifting the relative Frobenius diagram \eqref{eqn:frobdgm}:
\begin{equation} \label{eqn:lifted-frobdgm}
\xymatrix{
    \wt X \ar@/^1em/@{.>}[drr]^{\wt F_X} \ar@/_1em/[ddr]  \ar@{.>}[dr]|{\wt F_{X/S}} \\
    & \wt X{}' \ar[r] \ar[d] \ar@{}[dr]|\square & \wt X\ar[d] \\
    & \wt S \ar[r]_{\wt F_S=\sigma} & \wt S.
  }
\end{equation}
Given a lifting $\wt X$, the existence of $\wt F_X$ is thus equivalent to the existence of a lifting $\wt F_{X/S}\colon \wt X \to \wt X{}'$ of the relative Frobenius $F_{X/S} \colon X\to X'$. 

More generally, the above definitions work in the relative setting, that is when $S$ is an arbitrary $k$-scheme endowed with a Frobenius lifting $(\wt S, \wt F_S)$. If $\pi\colon X\to S$ is a scheme over $S$ and if $\swt\pi\colon \wt X \to \wt S$ is a lifting of $\pi$,  then providing a lifting $\wt F_X$ of $F_X$ to $\wt X$ such that $\wt F_S\circ \swt \pi = \swt\pi\circ\wt F_X$ is equivalent to providing a morphism $\wt F_{X/S}\colon \wt X\to \wt X{}'$ lifting the relative Frobenius $F_{X/S}$, where $\wt X{}'$ is the base change of $\wt X$ along $\wt F_S$.

If $(X, D)$ is an nc pair over $S$, a \emph{Frobenius lifting} of $(X, D)$ is a triple $(\wt X, \wt D, \wt F_X)$ where $(\wt X, \wt D)$ is an nc pair lifting $(X, D)$ and $\wt F_X$ is a lifting of $F_X$ to $X$ satisfying $\wt F{}_X^* (\wt D) = p\wt D$. 
We shall study this notion in more detail in \cite[Section 2.4]{PartII}. 
In view of Remark~\ref{rmk:nc-pair-vs-log-scheme}, Frobenius liftings of $(X, D)$ correspond to liftings of the associated log scheme $(X, \cM_X)$.   

\subsection{\texorpdfstring{Examples of $F$-liftable schemes}{Examples of F-liftable schemes}}
\label{ss:froblift-examples}

We shall be mostly concerned with smooth schemes in this paper. Examples of $F$-liftable singularities were studied in \cite{zdanowicz}.

\begin{example}[Smooth affines]
Every smooth affine $k$-scheme is $F$-liftable. Indeed, the obstruction class to lifting $X$ together with $F_X$ lies in $\Ext^1(\Omega^1_X,B^1_X) = H^1(X,T_X \otimes B^1_X)$ (see \cite[Appendix]{MehtaSrinivas}, cf.\ Proposition~\ref{prop:deformation_theory_frobenius}(b)), which is zero. 
\end{example}

\begin{example}[Toric varieties] \label{ex:toric}
Every toric variety $X=X(\Sigma)_k$ over $k$ is $F$-liftable \cite{BTLM}.  More precisely, let $\Sigma$ be a fan in $N_\RR$ for a lattice $N$ and let $X(\Sigma)$ be the associated toric variety over $\Spec\ZZ$. The multiplication by $p$ map $N\to N$ preserves $\Sigma$ and hence it induces a morphism $\wt F\colon X(\Sigma)\to X(\Sigma)$. Its restriction to $X(\Sigma)_{\FF_p}$ is the absolute Frobenius. 

Of course not every Frobenius lifting on a toric variety has to be of this type, e.g.\ any collection of homogeneous polynomials $f_0, \ldots, f_n\in k[x_0, \ldots, x_n]$ of degree $p$ defines a lifting of Frobenius on $\PP^n_{W_2(k)}$ by 
\[ 
  \wt F(x_0:\ldots : x_n) = (x_0^p + pf_0(x_0,\ldots, x_n): \ldots : x_n^p + pf_n(x_0,\ldots, x_n)).
\]
The existence of such non-standard liftings is one of the main difficulties in Conjecture~\ref{conj:froblift}, and provides a contrast between its two extreme (toric and abelian) cases.
\end{example}

\begin{example}[Ordinary abelian varieties] \label{ex:ordinaryav}
An abelian variety $A$ over $k$ is $F$-liftable if and only if it is ordinary, in which case there exists a unique Frobenius lifting $(\wt A, \wt F_A)$, called the Serre--Tate canonical lifting of $A$ (see \cite[Appendix]{MehtaSrinivas}). It has the remarkable property that for every line bundle $L$ on $A$, there exists a unique up to isomorphism line bundle $\wt L$ on $\wt A$ such that $\wt F{}_A^* \wt L \isom \wt L^{\otimes p}$.
\end{example}

\begin{example}[\'Etale quotients of ordinary abelian varieties] \label{ex:ordinaryav-etale}
Let $A$ be an ordinary abelian variety and let $f\colon A\to Y$ be a finite \'etale surjective morphism.  Replacing $f$ by its Galois closure, we can assume that it is Galois.  Then $Y$ is $F$-liftable by \cite[Theorem~2]{MehtaSrinivas}, but it need not be an abelian variety.
\end{example}

\begin{example}[Toric fibrations over ordinary abelian varieties] \label{ex:toric-fib-ordinaryav} 
In order to combine Examples~\ref{ex:toric} and \ref{ex:ordinaryav}, let $A$ be an abelian variety, and let $f \colon X \to A$ be a toric fibration under a torus $T$ (see \S\ref{ss:toric-var}). Then $X$ is $F$-liftable if and only if $A$ is ordinary (see Theorem~\ref{thm:frobenius_cy}).
\end{example}

\begin{remark}
In fact, in each of the above examples, the Frobenius liftings exist over $W(k)$ and not just over $W_2(k)$.  We do not know an example of a smooth $k$-scheme admitting a~Frobenius lifting only over $W_2(k)$.  However, it is important to work over $W_2(k)$ as this allows for descending Frobenius liftability under appropriate finite morphisms (see Theorem~\ref{thm:descending-frob-lift}).
\end{remark}

\begin{remark}[Partial converse to Conjecture~\ref{conj:froblift}] \label{rmk:partial-converse}
Consider the situation of the assertion of Conjecture~\ref{conj:froblift}: let $A$ be an ordinary abelian variety, let $Y\to A$ be a smooth toric fibration with toric boundary $D\subseteq Y$, and let $Y\to X$ be a finite \'etale map. Then $(Y, D)$ is $F$-liftable by Example~\ref{ex:toric-fib-ordinaryav}, and $X$ is $F$-liftable if one of the following conditions holds.
\begin{enumerate}[(1)]
  \item If $X = Y/G$ where $G$ is of order prime to $p$ (see Remark~\ref{remark:descend_reductive}). Note that since $\pi_1(Y) \isomto \pi_1(A)$, passing to the Galois closure we can always assume that $X=Y/G$ for a free action of a finite group $G$ on $Y$.
  \item If the toric boundary $D\subseteq Y$ is a pull-back from $X$ (Theorem~\ref{thm:frobenius_cy}(iii)$\Rightarrow$(i)), for example if $Y=A$ (Example~\ref{ex:ordinaryav-etale}).
\end{enumerate}
However, $X$ is not $F$-liftable in general: if $Y=\PP^1\times C$ where $C$ is an ordinary elliptic curve, and $G=\ZZ/p\ZZ$ acts on $\PP^1$ by $(x:y)\mapsto (x+ay:y)$ and on $C$ by translation, then the diagonal action of $G$ on $Y$ is free, and $X=Y/G$ is not $F$-liftable (see \cite[Proposition 3.6]{PartII} and \cite[Lemma 3.7]{PartII}). It would be interesting to find an `if and only if' criterion for the $F$-liftability of quotients $X=Y/G$ as above with $G$ an arbitrary finite group, or even an abelian $p$-group.
\end{remark}

\subsection{\texorpdfstring{Consequences of $F$-liftability}{Consequences of F-liftability}}
\label{ss:consequence_f_liftability}

As mentioned in the introduction, the existence of a Frobenius lifting has strong consequences for \emph{smooth} schemes.  First, let us recall the construction of the map $\xi$ (see \eqref{eqn:xi-intro}, cf. \cite[Proof of Th\'eor\`eme 2.1, (b)]{DeligneIllusie}). 

\subsubsection*{\texorpdfstring{Construction of the map $\xi$}{Construction of the map $\xi$}} 

Let $(\wt S, \wt F_S)$ be a Frobenius lifting of a $k$-scheme $S$, let $(\wt X, \wt F_{X})$ be a Frobenius lifting of a $k$-scheme $X$, and let $\swt \pi \colon \wt X \to \wt S$ be a smooth morphism of $W_2(k)$-schemes such that $\wt F_S \circ \swt \pi = \swt \pi \circ \wt F_X$. Let $\wt F_{X/S}$ be the induced lifting of the relative Frobenius.  By flatness of $\wt X \to \wt S$, the differential $d\wt F_{X/S}$ fits into the following diagram with exact rows 
\[
  \xymatrix{
    0 \ar[r] & F_{X/S}^*\Omega^1_{X'/S} \ar[r]\ar[d]^{dF_{X/S} = 0} & \wt F{}_{X/S}^*\Omega^1_{\wt X{}'/\wt S}\ar[r]\ar[d]^{d\wt F_{X/S}} & F_{X/S}^*\Omega^1_{X'/S} \ar[r]\ar[d]^{dF_{X/S} = 0} & 0 \\
    0 \ar[r] & \Omega^1_{X/S} \ar[r] & \Omega^1_{\wt X/\wt S} \ar[r] & \Omega^1_{X/S} \ar[r] & 0 
  }
\]
Using the snake lemma, we obtain a mapping 
\[
 \xi = \frac{d\wt F_{X/S}}{p} \colon F_{X/S}^*\Omega^1_{X'/S} \ra \Omega^1_{X/S}.
\]
In simple terms, $\xi(\omega) = \smash{\frac{1}{p} d\wt F_X(\swt \omega)}$, where $\swt \omega$ is any lifting of a local form ${\omega \in F_{X/S}^*\Omega^1_{X'/S}}$.  In particular, if $S = \Spec k$, then $\xi(df) = f^{p-1}df + dg$ where $\smash{\wt F{}_X^{*}(\wt{f})} = \smash{\wt f{}^p} + pg$.

Since $F_{X/S}^*\Omega^1_{X'/S} \isom F_X^*\Omega^1_{X/S}$, we may interpret $\xi$ as an $\cO_X$-linear morphism $F^*_X\Omega^1_{X/S} \to \Omega^1_{X/S}$.

\begin{prop}[{\cite[Proof of Th\'eor\`eme 2.1 and \S 4.1]{DeligneIllusie}, \cite[Theorem 2]{BTLM}}] \label{prop:frobenius_cotangent_morphism}
  The mapping $\xi \colon F_{X/S}^*\Omega^1_{X^{'}/S} \to \Omega^1_{X/S}$ satisfies the following properties:
  \begin{enumerate}[(a)]
    \item The adjoint morphism $\xi^{\rm ad}\colon \Omega^1_{X'/S} \to F_{X/S*}\Omega^1_{X/S}$ has image in the subsheaf $Z^1_{X/S}$ of closed forms and provides a splitting of the short exact sequence
    \[ 
      0\ra B^1_{X/S} \ra Z^1_{X/S}\overset{C_{X/S}}\ra \Omega^1_{X'/S} \ra 0.
    \]  
    \item By taking exterior powers, the morphism $\xi^{\rm ad}$ induces splittings of the short exact sequences
    \[ 
      0\ra B^i_{X/S} \ra Z^i_{X/S}\overset{C_{X/S}}\ra \Omega^i_{X'/S} \ra 0,
    \]  
    as well as a quasi-isomorphism of differential graded algebras
    \[ 
      \bigwedge^\bullet\, \xi^{\rm ad}\colon \bigoplus_{i\geq 0} \Omega^i_{X'/S}[-i]
      \isomto 
      (F_{X/S})_*\Omega^\bullet_{X/S}
    \] 
    where the maps $\bigwedge^i\, \xi^{\rm ad}\colon\Omega^i_{X'/S}\to (F_{X/S})_*\Omega^i_{X/S}$ are split injections. On the level of cohomology, this map induces the Cartier isomorphism.
    \item The determinant
    \[ 
      \det(\xi) \colon F_{X/S}^* \omega^1_{X'/S} \ra \omega^1_{X/S}
    \]
    corresponds to a Frobenius splitting $\sigma$ of $X$ relative to $S$ (see \S\ref{ss:f-split}). In particular, the homomorphism $\xi$ is injective.  
  \end{enumerate}
\end{prop}

\begin{variant}[Logarithmic variant of Proposition~\ref{prop:frobenius_cotangent_morphism}]\label{var:log_variant_xi}
If $(X, D)$ is an nc pair over $S$, and if $(\wt X, \wt D,\wt F_X)$ is a Frobenius lifting of $(X, D)$, we get a morphism 
\[ 
  \xi \colon F_{X}^*\Omega^1_{X/S}(\log D) \to \Omega^1_{X/S}(\log D)
\]
and the assertions of Proposition~\ref{prop:frobenius_cotangent_morphism} hold in this case (cf.\ \cite[\S 4.2]{DeligneIllusie}).
\end{variant}

\begin{cor}
  Let $X$ be an $F$-liftable smooth and proper scheme over $k$.  Then $X$ is ordinary in the sense of Bloch and Kato, i.e., $H^j(X, B^i_X) = 0$ for all $i, j\geq 0$. 
\end{cor}
\begin{proof}

This follows from the proof of \cite[Lemma 1.1]{MehtaSrinivas}.
\end{proof}

\subsubsection*{Bott vanishing}

By Proposition~\ref{prop:frobenius_cotangent_morphism}(b), we have $H^j(X, \Omega^i_X \otimes \cL) \subseteq H^j(X, \Omega^i_X \otimes \cL^p)$ for every line bundle $\cL$. This explains the following result. 

\begin{thm}[{Bott vanishing, \cite[Theorem 3]{BTLM}}] \label{thm:bott-vanishing}
  Let $X$ be a smooth projective $F$-liftable scheme over $k$, and let $\cL$ be an ample line bundle on $X$.  Then 
  \[ 
    H^j(X,\Omega^i_{X} \otimes \cL) = 0 \quad \text{for }j>0\text{, }i \geq 0.
  \]
  Moreover, if $(X, D)$ is an $F$-liftable nc pair, then \[ 
    H^j(X,\Omega^i_{X}(\log D) \otimes \cL) = 0 \quad \text{for }j>0\text{, }i \geq 0.
  \]
\end{thm}

\begin{cor} \label{cor:fliftable-fano-rigid}
  Let $X$ be a smooth projective $F$-liftable scheme over $k$. Suppose that $X$ is Fano (i.e., $\omega_{X}^{-1}$ is ample). Then $H^i(X, T_X) = 0$ for $i>0$. In particular, $X$ is rigid and has unobstructed deformations.
\end{cor}

\begin{proof}
Since $T_X \isom \Omega^{n-1}_{X} \otimes \omega_{X}^{\vee}$, Bott vanishing implies that
\[
  H^i(X, T_X) = H^i(X,\Omega^{n-1}_{X} \otimes \omega_{X}^{\vee}) = 0. \qedhere
\]
\end{proof}

The above results suggest that $F$-liftability is a rare property.  For instance, we obtain the following examples of non-liftable varieties.

\begin{example} \label{ex:hypersurfaces}
  A smooth hypersurface $X \subseteq \PP^n$ of degree $d>1$, where $n>1$, is not $F$-liftable as long as $n+d>5$. Indeed, we may reason as follows.
  \begin{enumerate}[(1)] 
    \item If $d>n+1$, then $X$ has positive Kodaira dimension, which contradicts Proposition~\ref{prop:frobenius_cotangent_morphism}(c).
    \item If $d=n+1$, then $\omega_X$ is trivial and $X$ is simply connected, contradicting \cite[Theorem 2]{MehtaSrinivas}.
    \item If $2\,{<}\,d\,{<}\,n+1$, then $X$ is Fano but not rigid, contradicting Corollary~\ref{cor:fliftable-fano-rigid}.
    \item If $d=2$ and $n>3$, then $X$ is not $F$-liftable because $H^1(X, \Omega^{n-2}_X(n-3)) \neq 0$, contradicting Bott vanishing (see \cite[\S 4.1]{BTLM}). See also \cite[Theorem~4.15]{zdanowicz}.
  \end{enumerate}
\end{example}

\subsubsection*{\texorpdfstring{The sheaf of $\xi$-invariant forms}{The sheaf of xi-invariant forms}}

Let $(\wt X, \wt F_X)$ be a Frobenius lifting of a smooth \mbox{$k$-scheme} $X$. We can regard the induced map 
\[
  \xi = \frac{1}{p}d\wt F_X \colon F_X^*\Omega^1_X\to \Omega^1_X
\]
as a Frobenius-linear endomorphism of $\Omega^1_X$, that is
\[
  \xi(f\cdot \omega) = f^p \cdot \xi(\omega) \ \text { for } \ f\in \cO_X,\ \omega\in \Omega^1_X.
\] 
As observed in \cite{MehtaSrinivas}, if $\omega_X$ is numerically trivial, then $\xi$ is an isomorphism, and therefore $\Omega^1_X$ becomes trivial on a finite \'etale cover $Y$ of $X$ (see \cite[Satz 1.4]{Lange_Stuhler}).

A related geometric idea, which we put to good use in \S\ref{s:homogeneous-spaces}, is to look at the (\'etale) subsheaf $(\Omega^1_X)^{\xi}$ of $\xi$-invariant forms in $\Omega^1_X$. By the very definition, $(\Omega^1_X)^{\xi}$ is the \'etale sheaf of sections of the fixed point locus $(T^* X)^{\xi}$ inside the cotangent bundle.

\begin{lemma} 
  The subscheme $(T^* X)^{\xi}\subseteq T^* X$ is \'etale over $X$ and the \'etale sheaf $(\Omega^1_X)^{\xi}$ is a~constructible sheaf of $\FF_p$-vector spaces on $X$. Over the dense open $U\subseteq X$ where $\xi$ is an isomorphism, $(T^* X)^{\xi}$ is finite over $X$ of degree $p^{\dim X}$ and $(\Omega^1_X)^{\xi}$ is an $\FF_p$-locally constant sheaf of rank $p^{\dim X}$.
\end{lemma}

\begin{proof}
The assertions are local on $X$, so we may assume that there exists an isomorphism $\cO_X^{\dim X} \isom \Omega^1_X$. Since $\xi$ as an endomorphism of $\Omega^1_X$ is Frobenius-linear, its fixed point locus is described by the following system of equations in variables $[f_1, \ldots, f_n]$:
\[
[f_1, \ldots, f_n] - [f_1^p, \ldots, f_n^p] \cdot A=0 ,
\]
for some $A \in M_{n\times n}(\Gamma(X, \cO_X))$. We immediately see that the Jacobian of the above system of equations is the identity matrix, and thus $(T^* X)^{\xi}$ is \'etale over $X$. The other assertions are clear in view of the fact that for a vector space $V$ of dimension $r$ endowed with a Frobenius-linear isomorphism $\xi \colon F^*V \to V$, the locus of fixed points $V^{\xi}$ is isomorphic to $\FF_p^r$. 
\end{proof}

\subsection{Deformation theory of the Frobenius morphism}
\label{ss:def-thy-frob}

In this subsection, we first discuss obstruction classes to lifting Frobenius, as originally defined by Nori and Srinivas.  Then we apply them, together with other tools, to show in Theorem~\ref{thm:descending-frob-lift} that under certain conditions on a morphism $\pi\colon Y\to X$, if $Y$ is $F$-liftable, then so is $X$.  We also discuss $F$-liftability of products, toric fibrations, and divisorial rings. 

\subsubsection*{Obstruction classes and their functoriality}

Since we are dealing with deformations of schemes and their Frobenii along the first order thickening $S\hookrightarrow \wt S$, it is natural to seek an obstruction theory for lifting the Frobenius. For smooth schemes over a perfect field $k$, such a theory was developed by Nori and Srinivas in the appendix to \cite{MehtaSrinivas}. The usual deformation theory of morphisms already tells us that
\begin{itemize}
  \item given a morphism $f\colon X\to Y$ and a lifting $\wt Y$, the obstruction to the existence of $\tilde f \colon \wt X \to \wt Y$ is a class in $\Ext^2(L_{X/Y}, \cO_X)$ {($\isom H^{2}(X, T_{X/Y})$ if $f$ is smooth)},
  \item if we also fix $\wt X$ a priori, then we have an obstruction lying in $\Ext^1(Lf^* L_{Y}, \cO_X)$ ($\isom H^1(X, f^* T_{Y})$ if $Y$ is smooth) to lifting $f$.
\end{itemize}
The first obstruction theory does not apply in our context of lifting $f=\smash{F_{X/k}} \colon X\to X'$ as the two liftings $\wt X$ and $\smash{\wt X{}'}$ of the source and the target are not independent, they are related by the pullback square in \eqref{eqn:lifted-frobdgm}. The second one does, and produces an obstruction to lifting the Frobenius to a~\emph{given} lifting $\wt X$. 

If we do not fix $\wt X$ a priori, then the classical theory is of no use, and in fact we cannot even write a suitable functor of Artin rings (one needs a lifting of Frobenius on the base). Nori and Srinivas bypass the difficulty by considering the problem of lifting $(X, F_X)$ only to $W_n(k)$, in which case the lifting $\sigma$ of Frobenius on $W_n(k)$ is unique. As in their situation, in \S\ref{s:log-mehta-srinivas} we shall actually need to lift not only to $W_2(k)$ but all the way to $W(k)$ to apply characteristic zero methods. 

Their results are as follows. 

\begin{prop}[{\cite[Appendix]{MehtaSrinivas}}]\label{prop:deformation_theory_frobenius}
  Let $X$ be a smooth scheme over $k$, and let $(X_n, F_{X_n})$ be a~Frobenius lifting of $X$ over $W_n(k)$. Then the following hold.
  \begin{enumerate}[(a)]
    \item For every lifting $X_{n+1}$ of $X_n$ over $W_{n+1}(k)$ there exists an obstruction class 
    \[
    o^F_{X_{n+1}} \in \Ext^1(\Omega^1_{X},F_{X*}\cO_X)
    \]
    whose vanishing is sufficient and necessary for the existence of a lifting $F_{X_{n+1}}$ of $F_{X_n}$ to $X_{n+1}$. If the obstruction vanishes, then the space of such liftings is a torsor under $\Hom(\Omega^1_{X},F_{X*}\cO_X)$. 
    \item There exists an obstruction class 
    \[
      o_{X_n} \in \Ext^1(\Omega^1_{X}, B^1_X)
    \]
    whose vanishing is sufficient and necessary for the existence of a lifting $(X_{n+1}, F_{X_{n+1}})$ of $(X_n, F_{X_n})$ over $W_{n+1}(k)$. If the obstruction vanishes, then the space of such liftings is a~torsor under $\Hom(\Omega^1_{X},B^1_X)$.
    \item The obstruction class $o_X \in \Ext^1(\Omega^1_{X}, B^1_X)$ to lifting $X$ over $W_2(k)$ compatibly with the Frobenius morphism equals the class of the extension
    \[
      0 \ra B^1_{X} \ra Z^1_{X} \xrightarrow{C_{X}} \Omega^1_{X} \ra 0.
    \]
    \item Let $(X_{n+1}, F_{X_{n+1}})$ be a lifting of $(X_n, F_{X_n})$ over $W_{n+1}(k)$ and suppose that the Frobenius action on $H^i(X,\cO_X)$ is bijective for $i = 1,2$. Then for every $\cL_n \in \Pic X_n$ such that $F_{X_n}^*\cL_n = \cL_n^{\otimes p}$, there exists a unique $\cL_{n+1} \in \Pic X_{n+1}$ such that $\cL_{n+1 | X_n} \isom \cL_n$ and $F_{X_{n+1}}^*\cL_{n+1} = \cL_{n+1}^{\otimes p}$.
  \end{enumerate} 
\end{prop}

\begin{variant}[Logarithmic variant of Proposition~\ref{prop:deformation_theory_frobenius}]\label{var:log_deformation_theory_frobenius}
The above proposition can be repeated word for word for nc pairs $(X, D)$. In this case, the sheaves of K\"ahler differentials are replaced by logarithmic differentials $\Omega^1_{X}(\log D)$.  The sheaf $B^1_{X}$ is left without change, since the \mbox{$1$-boundaries} $B^1_{X}(\log D)$ of the logarithmic de Rham complex coincide with the \mbox{$1$-boundaries} $B^1_{X}$ of the standard de Rham complex.  We do not give the full proof, but remark about the main ingredients of the logarithmic version of (b) and (c). 
\begin{enumerate}
  \item As in standard deformation theory, for an nc pair $(X_n,D_n)$ over $W_n(k)$ we notice that any two liftings over $W_{n+1}(k)$ (as nc pairs) are Zariski-locally isomorphic and the infinitesimal automorphisms are parametrized by sections of $\cHom(\Omega^1_X(\log D),\cO_X)$ (see \cite[Proposition 8.22]{esnault_viehweg}).
  \item The logarithmic analogue of Proposition~\ref{prop:deformation_theory_frobenius}(a) holds, and in particular for every Frobenius lifting $F_{X_n} \colon (X_n,D_n) \to (X_n,D_n)$ local liftings of the Frobenius morphism over $W_{n+1}(k)$ exist and are torsors under sections of $\cHom(\Omega^1_X(\log D),F_{X*}\cO_X)$ (see \cite[Proposition 9.3]{esnault_viehweg}).
\end{enumerate}
Now, to construct the obstruction class for lifting $(X_n,D_n,F_{X_n})$ over $W_{n+1}(k)$ we reason as follows.  We take an open affine covering $\{(U^i_n,D^i_n)\}$ of $(X_n,D_n)$, and then lift $(U^i_n,D^i_n)$ to some nc pairs $\smash{(U^i_{n+1},D^i_{n+1})}$ over $W_{n+1}(k)$.  By (2) the nc pairs $\smash{(U^i_{n+1},D^i_{n+1})}$ admit Frobenius liftings $\smash{F_{U^i_{n+1}} \colon (U^i_{n+1},D^i_{n+1}) \to (U^i_{n+1},D^i_{n+1})}$ extending $F_{X_n}|_{U^i_n}$.  We set $U^{ij}_{n+1} = U^i_{n+1} \cap U^j_{n+1}$.  By (1) we may fix isomorphisms 
\[
  \phi_{ij} \colon (U^j_{n+1},D^j_{n+1})_{|U^{ij}} \isomto (U^i_{n+1}, D^i_{n+1})_{|U^{ij}} \quad \text{satisfying} \quad \phi_{ji} = \phi_{ij}^{-1}.
\]  

The morphisms $F_{U^i_{n+1}}$ and $\phi_{ij} \circ F_{U^j_{n+1}} \circ \phi_{ji}$ are two liftings of $F_{X_n}|_{U^{ij}_n}$ and hence they give rise to local section $\tau_{ij}$ of the sheaf $\cHom(\Omega^1_X(\log D),F_{X*}\cO_X)$.  The images of $\tau_{ij}$ under the homomorphism $\cHom(\Omega^1_X(\log D),F_{X*}\cO_X) \to \cHom(\Omega^1_X(\log D),B^1_X)$ does not depend on the choices made, and give a cocycle whose cohomology class we claim is the required obstruction.  For the rest of the proof, we may repeat the reasoning of \cite[Appendix, Proposition 1 (vii)]{MehtaSrinivas} word for word.  For the proof of (c), we observe that the difference of the morphisms $\xi$ induced by the liftings $F_{\wt U{}^i}$ and $\phi_{ij} \circ F_{\wt U{}^j} \circ \phi_{ji}$ divided by $p$ is a map $h_{ij}\colon \Omega^1_{U^{ij}}(\log D) \to B^1_{U^{ij}}$, and that $\{h_{ij}\}$ is a cocycle representing the log Cartier sequence.
\end{variant}

The obstruction classes satisfy the following functoriality properties.

\begin{lemma} \label{lem:frob_functoriality_lifting}
  Let $\pi \colon Y \to X$ be a morphism of smooth schemes over $k$.
  \begin{enumerate}[(a)]
    \item Let $\swt \pi \colon \wt Y \to \wt X$ be a lifting of $\pi$ over $W_2(k)$. Then the obstruction classes $o^F_{\wt X}$ and $o^F_{\wt Y}$, treated as morphisms in the appropriate derived categories, fit into the following commutative diagrams
    \begin{equation}\label{eqn:functoriality-dgms-1}
      \xymatrix{
        \pi^*F_{X}^*\Omega^1_{X} \ar[r]^{\pi^*o^F_{\wt{X}}}\ar[d]_{F_{Y}^*d\pi} & \pi^*\cO_X[1] \ar@{=}[d] \ar@{}[drr]|-{\displaystyle \text{and}}  &  & F_{X}^*\Omega^1_{X} \ar[r]^{o^F_{\wt{X}}} \ar[d]_{F_Y^*d\pi} & \cO_X[1] \ar[d]^{\pi^{*}[1]} \\
        F_{Y}^*\Omega^1_{Y} \ar[r]_-{o^F_{\wt{Y}}} & \cO_Y[1]        & & R\pi_*F_{Y}^*\Omega^1_{Y} \ar[r]_-{R\pi_*o^F_{\wt{Y}}} & R\pi_*\cO_Y[1].
      }
    \end{equation}
    \item Suppose that $\pi$ is smooth. Then the obstruction classes $o_X$ and $o_Y$ fit into the following commutative diagram
    \[ 
      \xymatrix{
        \pi^* \Omega^1_{X} \ar[r]^{\pi^* o_X} \ar[d]_{d\pi} & \pi^* B^1_X[1] \ar[d]^{d\pi}\\
        \Omega^1_{Y} \ar[r]_{o_Y} & B^1_Y[1].
      }
    \]
  \end{enumerate}
\end{lemma}

\begin{proof}
For part (a), see \cite[Lemma 4.1]{zdanowicz}. For part (b), use Proposition~\ref{prop:deformation_theory_frobenius}(c) and the commutativity of the diagram 
\[
  \begin{gathered}[b] 
    \xymatrix{
      0 \ar[r] & \pi^* B^1_{X} \ar[r]\ar[d]_{d\pi} & \pi^* Z^1_{X} \ar[r]\ar[d]_{d\pi} & \pi^* \Omega^1_{X} \ar[r]\ar[d]_{d\pi} & 0 \\ 
      0 \ar[r] & B^1_{Y} \ar[r] & Z^1_{Y} \ar[r] & \Omega^1_{Y} \ar[r] & 0. 
    }\\[-\dp\strutbox] 
  \end{gathered}
  \qedhere
\]
\qedhere 
\end{proof}

\begin{remark} \label{remark:cotangent_complex}
In the case of singular schemes, the cotangent bundle $\Omega^1_{X}$ can be substituted with its derived variant $L_{X/k} \in D^-_{\rm Coh}(\cO_X)$ (see \cite{illusie_cotangent}).  The assertions of Proposition~\ref{prop:deformation_theory_frobenius}(a) and Lemma~\ref{lem:frob_functoriality_lifting}(a) remain valid with the standard differential replaced by the derived one.
\end{remark}

\subsubsection*{Descending and lifting Frobenius liftings}

With the abstract deformation theory at hand, we can now easily relate $F$-liftability of two schemes $X$ and $Y$ in the presence of a suitable morphism $\pi \colon Y \to X$.  In the subsequent sections, we shall frequently use the following results.

\begin{lemma} \label{lemma:ascending-frob-lift}
  Let $\pi \colon Y\to X$ be an \'etale morphism of $k$-schemes. For every Frobenius lifting $(\wt X, \wt F_X)$ of $X$, there exists a unique Frobenius lifting $(\wt Y, \wt F_Y)$ of $Y$ and a lifting $\swt \pi\colon \wt Y\to \wt X$ of $\pi$ such that $\wt F_X\circ \swt \pi = \swt \pi\circ \wt F_Y$.
\end{lemma}

\begin{proof}
This follows from the equivalence of categories between \'etale schemes over $X$ and over~$\wt X$. 
\end{proof}

\begin{thm}[Descending Frobenius liftability] \label{thm:descending-frob-lift}
  Let $\pi \colon Y\to X$ be a morphism of schemes (essentially) of finite type over $k$ and let $(\wt Y, \wt F_{Y})$ be a Frobenius lifting of $Y$.
  \begin{enumerate}[(a)]
    \item Suppose that $\pi$ admits a lifting $\swt \pi\colon \wt Y\to \wt X$, and that one of the following conditions is satisfied:
    \begin{enumerate}[i.]
        \item $\pi^*\colon \cO_X \to R\pi_* \cO_Y$ is a split monomorphism in the derived category,
        \item $\pi$ is finite flat of degree prime to $p$,
        \item $Y$ satisfies condition $S_2$ and $\pi$ is an open immersion such that $X\setminus Y$ has codimension $>1$ in $X$.
    \end{enumerate}
    Then $F_X$ lifts to $\wt X$.  
    \item Suppose that one of the following conditions is satisfied:
      \begin{enumerate}[i.]
        \item $\cO_X\isomto \pi_* \cO_Y$ and $R^1\pi_* \cO_Y = 0$,
        \item $X$ and $Y$ are smooth and $\pi$ is proper and birational,
        \item $Y$ satisfies condition $S_3$ and $\pi$ is an open immersion such that $X\setminus Y$ has codimension $>2$ in $X$.
      \end{enumerate}
    Then there exists a unique pair of a Frobenius lifting $(\wt X, \wt F_X)$ of $X$ and a lifting $\swt \pi \colon \wt Y\to \wt X$ of $\pi$ such that $\wt F_X\circ \swt \pi = \swt \pi\circ \wt F_Y$. 
  \end{enumerate}
\end{thm}

In fact, conditions (a.ii) and (b.ii) imply (a.i) and (b.i), respectively. We do not expect $\wt F_X\circ \swt \pi = \swt \pi\circ \wt F_Y$ to hold in general in situation (a). 

\begin{remark}
\label{remark:descend_reductive}
In the sequel \cite[Section 2.2]{PartII} we also prove the following result: 
{\it 
\begin{enumerate}[(a)]
\item[(c)] Suppose that $Y$ is normal and that $\pi \colon Y\to X=Y/G$ is a good quotient by an action of a~linearly reductive group $G$ on $Y$. Then there exists a lifting $\swt \pi \colon \wt Y\to \wt X$ of $\pi$ and a lifting $\wt F_X$ of $F_X$ to $\wt X$.  
\end{enumerate}
}
\noindent The approach is to use (a) for a map $\swt \pi$ obtained using averaging technique for Frobenius splittings.
\end{remark}

\begin{proof}
(a) Under condition (i), the right arrow of the right diagram \eqref{eqn:functoriality-dgms-1} is a split injection by assumption. Thus $o_{\wt X}^F = 0$ if $o_{\wt Y}^F=0$. Condition (ii) implies (i), as $R\pi_*\cO_Y=\pi_*\cO_Y$ and $1/\deg(\pi)$ times the trace map yields a splitting. For (iii), we argue as in \cite[Corollary 4.3]{zdanowicz}. Let $K$ be the fiber of the right arrow in the right diagram \eqref{eqn:functoriality-dgms-1}, fitting into an exact triangle
\begin{align}
  K \ra \cO_X[1] \ra R\pi_* \cO_Y[1] \ra K[1] \label{eqn:local_cohomology}.
\end{align}
Since the bottom map in the right diagram \eqref{eqn:functoriality-dgms-1} is zero by assumption, the top map has to factor through $K$.  It is therefore enough to show that $\Hom(F^*_X \Omega^1_X, K)=0$. Note that $K = R\Gamma_Z(\cO_X)[1]$ is the shift by one of the local cohomology complex with supports on $Z=X\setminus Y$ (see \cite{hartshorne_local_cohomology} or \stacksproj{0A39}).  Considering the spectral sequence 
\[
  E^{p,q}_2 = \Hom\left(F^*_X \Omega^1_X,\cH^q(R\Gamma_Z(\cO_X))[p]\right) 
  \quad \Rightarrow \quad
  \Hom\left(F^*_X \Omega^1_X,R\Gamma_Z(\cO_X)[p+q]\right),
\]
for $p+q = 1$, we see that it suffices to show that the local cohomology vanishes up to degree one, which is implied by the Serre's condition $S_2$ and \cite[Proposition 3.3]{hassett_kovacs}.  Analogous reasoning works if $X$ and $Y$ are not smooth.  In this case, as mentioned in Remark~\ref{remark:cotangent_complex}, we substitute the sheaf of K\"ahler differentials with the cotangent complex.

(b)  Under condition (i), Lemma~\ref{lem:liedtke_satriano} provides a lifting $\swt \pi \colon \wt Y\to \wt X$ defined by the assignment $\cO_{\wt X} = \pi_*\cO_{\wt Y}$.  To obtain a Frobenius lifting on $\wt X$ we just take $F_{\wt X} = \pi_*F_{\wt Y}$.  For (ii), we observe that, since $X$ and $Y$ are smooth, \cite[Theorem 1.1]{CR15} implies that $R^i\pi_*\cO_Y = 0$ for $i > 0$, and hence we may use (i) to conclude.  For (iii), we reason similarly as in (a.iii).   More precisely, we consider the long exact sequence of cohomology for \eqref{eqn:local_cohomology} to see that $R^1\pi_*\cO_Y$ is isomorphic to the second local cohomology supported in $Z = X \setminus U$, which vanishes by condition $S_3$ and \cite[Proposition 3.3]{hassett_kovacs}.
\end{proof}

\begin{cor} \label{cor:flift-products}  
  Let $X$ and $Y$ be smooth and proper schemes over $k$.  Then $X\times Y$ is \mbox{$F$-liftable} if and only if $X$ and $Y$ are. 
\end{cor}

\begin{proof}
If $(\wt X,\wt F_X)$ and $(\wt Y,\wt F_{Y})$ are Frobenius liftings of $X$ and $Y$, respectively, then ${(\wt X \times \wt Y}$, ${\wt F_X \times \wt F_Y)}$ is a Frobenius lifting of $X\times Y$.  For the converse, we first use the arguments of \cite[Lemma 1]{BTLM}. The sheaf $\Omega^1_{X \times Y}$ decomposes as a direct sum $\pi_X^*\Omega^1_{X} \oplus \pi_Y^*\Omega^1_{Y}$ and therefore the morphism of de Rham complexes $\Omega^\bullet_{X} \to \pi_{X*}\Omega^\bullet_{X \times Y}$ induced by the differential $d\pi_X \colon \Omega^1_{X} \to \Omega^1_{X \times Y}$ has a natural splitting.  This leads to a splitting $s$ of the morphism of short exact sequences
\[
\xymatrix{
  0 \ar[r] & B^1_X \ar[r]\ar[d]_{d\pi_X} & Z^1_{X} \ar[r]\ar[d]_{d\pi_X} & \Omega^1_{X} \ar[r]\ar[d]_{d\pi_X} & 0 \\ 
  0 \ar[r] & \pi_{X*}B^1_{X \times Y} \ar[r]\ar@/_0.5pc/[u]_s & \pi_{X*}Z^1_{X \times Y} \ar[r]\ar@/_0.5pc/[u]_s & \pi_{X*}\Omega^1_{X \times Y} \ar[r]\ar@/_0.5pc/[u]_s & 0.
}
\]
Consequently, we see that the upper row is split if the lower row is. This completes the proof. 
\end{proof}

\subsubsection*{Toric fibrations}
The goal of this part of the section is to show that split toric fibrations over an $F$-liftable base are $F$-liftable (see Example~\ref{ex:toric-fib-ordinaryav}). We remove the assumption that the fibration is split in Theorem~\ref{thm:frobenius_cy}.

Let $X$ be a normal $k$-scheme, and let $L_1, \ldots, L_n$ be line bundles on $X$. Consider the graded $\cO_X$-algebra 
\begin{equation} \label{eqn:def-r-sheaf} 
  \cR = \bigoplus_{\lambda\in \ZZ^n} \cR_\lambda, \quad
  \cR_\lambda = L_1^{\lambda_1}\otimes\ldots\otimes L_n^{\lambda_n},
\end{equation}
with multiplication given by the tensor product, and set $U=\Spec_X \cR$. The natural map $U\to X$ is a torsor under the split torus $T=\GG_m^n$. Let $Z=X(\Sigma)_k$ be a toric variety under the torus $T$ and let $Y= U\times^T Z$. The projection $\pi\colon Y\to X$ is a split toric fibration with fiber $Z$, and conversely every split toric fibration over $X$ arises via this construction.

\begin{lemma} \label{lemma:flift-torsors}
  In the above situation, let $(\wt X, \wt F_X)$ be a Frobenius lifting of $X$. Then there exists a~Frobenius lifting $(\wt Y, \wt F_Y)$ and a lifting $\swt\pi\colon \wt Y\to \wt X$ of $\pi$ such that $\wt F_X\circ \swt \pi = \swt \pi \circ \wt F_Y$.
\end{lemma}

\begin{proof}
Let $\wt L_i$ be the unique liftings of $L_i$ to $\wt X$ satisfying $\wt F{}_X^* \wt L_i \isom \wt L_i^p$, which exist by Proposition~\ref{prop:deformation_theory_frobenius}(d) and Lemma~\ref{lem:f-split_is_ordinary}(a). We start with the case $Z=\GG_m^n$, so $Y=U$. Consider the graded $\cO_{\wt X}$-algebra 
\[ 
  \wtcR = \bigoplus_{\lambda \in \ZZ^n} \wtcR_\lambda,
  \quad \wtcR_\lambda =  \wt L_1^{\lambda_1}\otimes\ldots\otimes \wt L_n^{\lambda_n},
\]
and set $\wt U = \Spec_X \wtcR$. The natural map $\swt\pi\colon\wt U\to \wt X$ is a $\GG_m^n$-torsor lifting $\pi\colon U\to X$. Moreover, the map
\[ 
  \wt F{}_X^* \wtcR = \bigoplus_{\lambda \in \ZZ^n} \wt F{}_X^* \wtcR_\lambda \isom \bigoplus_{\lambda \in \ZZ^n} (\wtcR_\lambda)^{\otimes p} = \bigoplus_{\lambda \in p\ZZ^n} \wtcR_{\lambda} \hookrightarrow \bigoplus_{\lambda \in \ZZ^n} \wtcR_\lambda = \wtcR
\]
induces a map $\wt F_{U/X}\colon \wt U\to \wt U{}'$, which lifts the relative Frobenius $F_{U/X}\colon U\to U'$, where $\wt U{}'$ is the base change of $\wt U$ along $\wt F_X$ and $U'$ is its reduction modulo $p$.

For the general case, we set 
\[  
  \wt Y=\wt U \times^{\wt T} X(\Sigma)_{W_2(k)} \overset{\swt\pi}{\ra} \wt X, \quad \text{ where }\wt T = \GG^n_{m, W_2(k)}, 
\]
the toric fibration with fiber $X(\Sigma)$ associated to $\wt U$. This is a lifting of $Y$, and the lifting of Frobenius on $\wt U$ extends to $\wt Y$.
\end{proof}

\section{Toric varieties in families}
\label{s:toric-reductions}

In this section, we address the following three questions:
\begin{itemize}
  \item {\bf Generalization:} Given a family $f\colon X\to S$ such that $X_s$ is a toric variety for a~dense set of $s \in S$, can we deduce that the generic fiber is a toric variety? (\S\ref{ss:generalization})
  \item {\bf Specialization:} Given a family $f\colon X\to S$ whose generic fiber is a toric variety and $S = \Spec R$ for a discrete valuation ring $R$, when can we deduce that the special fiber is a toric variety as well? (\S\ref{ss:specialization})
  \item {\bf Global rigidity:} Given a proper nc pair $(X, D)$ over a connected scheme $S$, if one geometric fiber $(X, D)_{\bar s}$ is a toric pair, must $(X, D)$ globally come from a toric fibration? (\S\ref{ss:global-rigidity})
\end{itemize}
The result of \S\ref{ss:generalization} will be used in the subsequent \S\ref{ss:images} to show that Conjecture~\ref{conj:jarekw} follows from (a~special case of) Conjecture~\ref{conj:froblift}.  The results of \S\ref{ss:specialization}--\S\ref{ss:global-rigidity} will be needed in Section~\ref{s:log-mehta-srinivas}.

\subsection{Generalization}
\label{ss:generalization}

We start with a somewhat lengthy proof of the following fact which should be well-known but for which we were unable to find a reference.

\begin{lemma} \label{lem_isom_on_Pic}
  Let $S$ be a geometrically unibranch \cite[6.15.1]{EGAIV_2} noetherian scheme, and let $\pi \colon X\to S$ be a smooth proper morphism whose geometric fibers are connected and satisfy 
  \begin{equation} \label{eqn:vanishing-h1h2} 
    H^1(X_{\bar s}, \cO_{X_{\bar s}})=H^2(X_{\bar s}, \cO_{X_{\bar s}})=0.
  \end{equation}
  Then there exists a finite \'etale surjective morphism $S'\to S$ such that $\Pic_{X'/S'}$ is the constant sheaf associated to a finitely generated group on the big \'etale site of $S'$, where $X'$ is the base change of $X$ to $S'$.
\end{lemma}

\begin{proof} 
By \cite[Theorem~9.4.8]{Kleiman}, $\Pic_{X/S}$ is representable by the disjoint union of quasi-projective schemes over $S$.  By the deformation theory of line bundles, $H^2(X_{\bar s}, \cO_{X_{\bar s}})$ is the obstruction space and $H^1(X_{\bar s}, \cO_{X_{\bar s}})$ is the tangent space of the deformation functor of a line bundle on $X_{\bar s}$, and hence \eqref{eqn:vanishing-h1h2} shows that line bundles deform uniquely over thickenings of $X_{\bar s}$. Consequently, $\Pic_{X/S}$ is formally \'etale, and hence \'etale, over $S$. Moreover, $\Pic_{X/S}\to S$ satisfies the valuative criterion of properness. Thus if $P$ is a connected component of $\Pic_{X/S}$, then $P\to S$ is a connected \'etale covering of $S$, and since $S$ is geometrically unibranch, $P$ is finite over $S$. We conclude that $\Pic_{X/S}$ is the disjoint union of connected finite \'etale coverings of $S$. 

We assume without loss of generality that $S$ is connected, and pick a geometric point $\bar s$ of $S$. Let $M=\Pic_{X/S}(\bar s) = \Pic(X_{\bar s})$, which is a finitely generated abelian group. Pick a~finite set of generators $p_1, \ldots, p_k \in M$, and for each $i=1, \ldots, k$, let $P_i$ be the connected component of $\Pic_{X/S}$ containing the corresponding geometric point $\bar p_i\to \Pic_{X/S}$. Let $S'\to S$ be a finite \'etale cover after pullback to which each $P_i\to S$ becomes constant. Replacing $S$ by $S'$, we can assume that each $P_i$ maps isomorphically onto $S$, that is, each $p_i\in \Pic_{X/S}(\bar s)$ is the restriction of a (unique) global section $q_i \in \Pic_{X/S}(S)$. We claim that in this case $\Pic_{X/S}$ is actually constant. Let $P = M\times S$ be the constant group scheme over $S$ associated to $M$; we will construct an isomorphism $\Pic_{X/S}\isomto P$. The sections $p_i$ and $q_i$ define surjective morphisms of group schemes over $S$:
\[ 
  \alpha \colon \ZZ^k\times S\ra \Pic_{X/S} \quad\text{and}\quad \beta \colon \ZZ^k\times S \ra P.
\]
Moreover if the $p_i$ satisfy a relation $\sum a_ip_i =0$,  then so do the sections $q_i$, and hence there is a~surjective morphism $\gamma \colon \Pic_{X/S}\to P$. Its kernel $K$ is a closed subscheme of $\Pic_{X/S}$ which is flat, and hence \'etale, over $S$. Moreover, $K_{\bar s} = \{1\}$ by construction, and hence $K\isomto S$. Thus $\gamma$ is an isomorphism.
\end{proof}

The above assertion need not be true when $H^2(X_{\bar s}, \cO_{X_{\bar s}})\neq 0$, for example for non-isotrivial families of K3 surfaces over complete complex curves \cite{BorcherdsK3}.

\begin{prop} \label{prop:split}
  Let $S$ be a noetherian excellent scheme and let $\pi \colon X\to S$ be a smooth projective morphism whose geometric fibers are connected and satisfy \eqref{eqn:vanishing-h1h2}. Let $\cE$ be a~locally free sheaf of rank $r$ on $X$.  Then the set
  \[ 
    \left\{ s\in S\, \middle|\, \cE_{\bar s} \text{ is a direct sum of line bundles on }X_{\bar s} \text{ for every geom.\ pt.\ } \bar s \text{ over }s \right\} \subseteq S
  \] 
  is a constructible subset of $S$. 
\end{prop}

\begin{proof}
Stratifying $S$, we may assume that $S$ is connected and regular. By Lemma~\ref{lem_isom_on_Pic}, after replacing $S$ with some finite \'etale cover $S'\to S$, we have $\Pic_{X/S} \isom M\times S$ for a finitely generated group $M$, and hence $\Pic(X_s)=\Pic(X_{\bar s}) = M$ for all $s\in S$. Given $H, D\in M$, there is a well-defined intersection number $H^{d-1}\cdot D\in \ZZ$, the same for all fibers, where $d$ is the relative dimension of $f$. 

We can assume that there exist $H_1, \ldots H_s \in M$ which give ample line bundles on every fiber and which span $M\otimes\bb Q$. If $H$ is one of these, then by Noetherian induction there is a natural number $n_H$ such that $\cE(n_H H)$ is globally generated on all fibers. If $L\in M$ is a~direct summand of $\cE_s$, it follows that $L + n_H H$ is a direct summand of a globally generated sheaf, and hence is effective on $X_s$, so $(L+n_H H)\cdot H^{d-1}\geq 0$ (independent of $s$). Thus $L\cdot H^{d-1} \geq -n_H H^d$. Applying the same argument to $\cE^\vee$, we get a natural number $m_H$ such that $-L \cdot H^{d-1} \geq -m_H H^d$, i.e.\ $L\cdot H^{d-1}\leq m_H H^d$. We conclude that there exists a natural number $B$ such that 
\begin{equation}\label{eqn:ineq} 
  |L\cdot H^{d-1}_i|<B \quad \text{for} \quad  i=1,\ldots, s 
\end{equation}
whenever $L$ is a direct summand of some $\cE_s$. 

Because $H_i$ span $M\otimes \bb Q$, the set $K$ of all $L\in M$ satisfying the inequalities \eqref{eqn:ineq} is finite.  For $v = (L_1, \ldots, L_r) \in K^r$, consider the functor
\[ 
  F_v \colon (\cat{Schemes}/S) \ra \cat{Sets}, \quad
  T\mapsto {\rm Isom}_{X\times_S T}((\oplus_{i=1}^r L_i)_T, \cE_T).
\]
Stratifying $S$ further, we may assume that $\pi_*\cHom((\oplus_{i=1}^r L_i), \cE)$ is locally free and its formation commutes with base change.  This implies that $F_v$ is representable by an open subscheme of the total space of the vector bundle $\pi_*\cHom((\oplus_{i=1}^r L_i), \cE)$, and hence is a scheme of finite type over $S$.  Let $W$ be the disjoint union of these schemes, which is still of finite type as $K^r$ is a finite set. The locus where $\cE$ splits is the set-theoretic image of $W\to S$, which is constructible by Chevalley's theorem.
\end{proof}

We shall now apply this to the study of toric varieties in families. As observed by Jaczewski and developed further by K\k{e}dzierski and Wi\'sniewski, smooth toric varieties admit a generalization of the Euler sequence, and this property in fact characterizes smooth projective toric varieties, at least in characteristic zero.

\begin{defin}[{cf.\ \cite[Definition~2.1]{Jaczewski} and \cite[\S 1.1]{KedzierskiWisniewski}}]
  Let $\pi\colon X\to S$ be a smooth projective morphism of schemes. Suppose that $S$ is affine, and that the coherent sheaf $\cH = R^1 \pi_* \Omega^1_{X/S}$ is locally free. Then the image of ${\rm id}_\cH$ under the natural identification
  \begin{align*}
    \Hom_S(\cH, \cH) &\isom H^0(S, R^1\pi_*(\pi^* \cH^\vee\otimes \Omega^1_{X/S})) \\
    &\isom H^0(S, \cExt^1_{X/S}(\pi^*\cH, \Omega^1_{X/S})) \isom \Ext^1(\pi^* \cH, \Omega^1_{X/S})
  \end{align*}
  gives rise to an extension
  \begin{equation}  \label{eqn:euler-jacz}
    0\ra \Omega^1_{X/S} \ra R_{X/S} \ra \pi^* \cH \ra 0.
  \end{equation}
  The locally free sheaf $R_{X/S}$ is called the \emph{potential sheaf} of $X$ over $S$.
\end{defin}

\begin{thm}[{\cite[Theorem~1.1 and Corollary~2.9]{KedzierskiWisniewski}}] \label{thm:kedz-wisn}
  Let $K$ be an algebraically closed field, and let $X$ be a smooth projective integral scheme over $K$. If $X$ is a toric variety, then the sheaf $R_{X/K}$ splits into the direct sum of line bundles. Conversely, if $K$ has characteristic zero, the potential sheaf $R_{X/K}$ splits into the direct sum of line bundles, and we have the vanishing
  \[ 
    H^1(X, \cO_X) = H^2(X, \cO_X) = 0,
  \]
  then $X$ is a toric variety.
\end{thm}

\begin{cor} \label{cor:toric-generalization}
  Let $S$ be an integral noetherian scheme with generic point $\eta$ of characteristic zero, and let $\pi \colon X\to S$ be a smooth and projective morphism. Suppose that the geometric fiber $X_{\bar s}$ is a~toric variety for a dense set of closed points $s\in S$. Then the geometric generic fiber $X_{\bar \eta}$ is a toric variety.
\end{cor}

\begin{proof}
Shrinking $S$, we may assume that $S$ is affine and that $\smash{R^1 \pi_* \Omega^1_{X/S}}$ is locally free, in which case the potential sheaf $R_{X/S}$ is defined. Moreover, since $H^0(Y, \cO_Y)=k$ and $H^i(Y, \cO_Y)=0$ for any $i>0$ and a toric variety $Y$ over a field $k$, we see that the assumption on $H^i(X_{\bar s}, \cO_{X_{\bar s}})$ in Proposition~\ref{prop:split} is satisfied.  We apply this result to the potential sheaf $R_{X/S}$. Since $R_{Y/k}$ splits for a toric variety $Y$ by the first part of Theorem~\ref{thm:kedz-wisn}, we deduce that $R_{X_{\bar\eta}/\bar\eta} = (R_{X/S})_{\bar\eta}$ must be split as well. By the other direction of Theorem~\ref{thm:kedz-wisn}, this implies that $X_{\bar\eta}$ is a toric variety. 
\end{proof}

\subsection{Specialization}
\label{ss:specialization}

Let $R$ be a discrete valuation ring with residue field $k$ and fraction field $K$. We set
\[ 
  S=\Spec R, \quad s = \Spec k, \quad \bar s = \Spec \bar k, \quad \eta = \Spec K, \quad \bar\eta = \Spec\overline K.
\]
Suppose that $X$ is a smooth projective scheme over $S$ whose general fiber $X_\eta$ is a toric variety. It is not true in general that the geometric special fiber $X_{\bar s}$ is a toric variety, as the following basic example shows.

\begin{example}
Let $s_0, s_1, s_2 \in \PP^2(R)$ be three sections over $S$ which give a triple of distinct collinear points in the special fiber $\PP^2(k)$, but which are not collinear in $\PP^2(K)$. Let $X$ be the blow-up of $\PP^2_S$ along the union of these three sections. Then $X$ is a smooth projective surface over $S$, and $X_\eta$ is a~toric variety but $X_{\bar s}$ is not.
\end{example}

The above phenomenon cannot happen if we consider deformations of toric varieties together with their toric boundaries. The goal of this subsection is to show that toric pairs can only degenerate to toric pairs.

\begin{prop} \label{prop:toric-specialization} 
  Let $X$ be a smooth proper scheme over $S$, and let $D\subseteq X$ be a divisor with normal crossings relative to $S$. If $(X_\eta, D_\eta)$ (resp.\ $(X_{\bar\eta}, D_{\bar\eta})$) is a toric pair (see \S\ref{ss:toric-var}), then so is $(X_s, D_s)$ (resp.\ $(X_{\bar s}, D_{\bar s})$).
\end{prop}

\begin{proof}
This proof is inspired by \cite[Theorem~I~2.1]{AMRT}. Suppose that we have an isomorphism $g\colon (X, D)_\eta \isomto (X(\Sigma), D(\Sigma))_\eta$ for some fan $\Sigma$. We will show the stronger claim that $(X, D) \isom (X(\Sigma), D(\Sigma))_S$. Every $m\in M$ (the monomial lattice of $\Sigma$) defines a rational function $g^*(m)$ on $X$ defined and invertible outside of $D\cup X_s$. Their valuations $\nu_s(g^*(m)) \in \ZZ$ along the prime divisor $X_s\subseteq X$ define a homomorphism $M\to \ZZ$, i.e.\ an element $\gamma$ of the dual lattice $N$. If $\pi\in R$ is a uniformizer, i.e.\ $\nu_s(\pi)=1$, we define an element $\tau$ of the torus $T(K) = \Hom(M, K^\times)$ defined by $m\mapsto \pi^{-\gamma(m)}$. If we now replace $g$ with the composition $h= \tau\circ g \colon (X, D)_\eta\to (X(\Sigma), D(\Sigma))_\eta$, we obtain another isomorphism as before but now with $\gamma = 0$; in other words, the rational functions $h^* (m)$ ($m\in M$) on $X$ are defined and invertible at the generic point of the special fiber $X_s$. 

We now check that the above condition implies that $h$ extends to a map $(X, D)\to (X(\Sigma), D(\Sigma))_S$. For a ray $\rho\in \Sigma(1)$, let $D(\rho)\subseteq X(\Sigma)$ be the corresponding component of $D(\Sigma)$, and let $D_\rho \subseteq X$ be the component of $D$ containing $h^{-1}(D(\rho)_\eta)$. For a cone $\sigma \in \Sigma$, let $U(\sigma) = X(\Sigma) \setminus \bigcup_{\rho \not\subseteq \sigma} D(\rho)$ be the corresponding affine toric variety; similarly, we define $U_\sigma = X \setminus \bigcup_{\rho \not\subseteq \sigma} D_\rho$, so that $(U_\sigma)_\eta = h^{-1}(U(\sigma)_\eta)$. Note that the $U_\sigma$ cover $X$, as all intersections of components of $D$ are flat over $S$ and the $(U_\sigma)_\eta$ cover $X_\eta$. Further, the condition $\gamma=0$ implies that each $h|_{U_\sigma}\colon (U_\sigma)_\eta\to U(\sigma)$ extends to $U_\sigma$. Indeed, the affine toric variety $U(\sigma)$ represents the functor $T\mapsto \Hom_{\rm monoids}(\sigma^\vee\cap M, \Gamma(T, \cO_T))$, and the map we want to extend is induced by $h^*$ on $M$, so all we want is to extend the functions $h^*(m)$ on $(U_\sigma)_\eta$ ($m\in \sigma^\vee\cap M$) across the divisor $X_s \cap U_\sigma$ and this is ensured by $\gamma=0$. Therefore  $h$ extends to a map $X = \bigcup U_\sigma \to \bigcup U(\sigma)_S = X(\Sigma)_S$ mapping $D$ into $D(\Sigma)_S$ since this holds on $X_\eta$.

Finally, we check that the map $X\to X(\Sigma)$ is an isomorphism. In fact, every map $f\colon X\to Y$ with $X$ and $Y$ smooth and proper over $S$ with connected fibers such that $f_\eta$ is an isomorphism has to be an isomorphism. Indeed, by van der Waerden's purity theorem, the exceptional locus of $f$ is a divisor in $X$ contained in $X_s$, so if non-empty, it equals $X_s$. But in that case, $\dim {\rm im}(f)<\dim X_s$, contradicting the semi-continuity of fiber dimension applied to the proper map ${\rm im}(f)\to S$.  

To deduce the claim about geometric fibers, if $(X_{\bar\eta}, D_{\bar\eta})$ is a toric pair, then there exists a finite extension $\eta'/\eta$ such that $(X_{\eta'}, D_{\eta'})$ is toric. Let $S'\to S$ be the normalization of $S$ in $\eta'$, and let $s'$ be the closed point of $S'$. Applying the above to the base change $X'\to S'$, we deduce that $(X_{s'}, D_{s'})$ is toric, and hence so is $(X_{\bar s}, D_{\bar s})$.
\end{proof}

\subsection{Global rigidity of toric pairs}
\label{ss:global-rigidity}

The following result shows that there are no interesting families of toric pairs. 

\begin{prop} \label{prop:global-rigidity}
  Let $f\colon X\to S$ be a smooth and projective morphism to a non-empty noetherian connected scheme $S$ and let $D\subseteq X$ be a divisor with normal crossings relative to $S$. The following conditions are equivalent:
  \begin{enumerate}[(i)]
    \item The morphism $f$ admits a structure of a toric fibration with toric boundary $D$.
    \item \'Etale-locally on $S$, there exists a smooth fan $\Sigma$ and an $S$-isomorphism 
    \[ 
      (X, D)\isom (X(\Sigma), D(\Sigma))_S.
    \] 
    \item For a single geometric point $\bar s$ of $S$, there exists a smooth fan $\Sigma$ and an isomorphism
    \[ 
      (X, D)_{\bar s} \isom (X(\Sigma), D(\Sigma))_{\bar s}.
    \]
  \end{enumerate}
\end{prop}

\begin{proof}
Implications (i)$\Rightarrow$(ii)$\Rightarrow$(iii) are obvious. For (iii)$\Rightarrow$(ii), it is enough to consider $S$ irreducible. Consider the scheme 
\[
  I={\rm Isom}_S((X, D), (X(\Sigma), D(\Sigma))_S),
\]
which is locally of finite type over $S$. If $\bar i$ is a geometric point of $I$ and $\bar s$ is its image in $S$, then $\Omega^1_{X_{\bar s}/\bar s}(\log D) \isom \Omega^1_{X(\Sigma)_{\bar s}}(\log D(\Sigma)_{\bar s})$ is trivial and $H^i(X_{\bar s}, \cO_{X_{\bar s}}) \isom H^i(X(\Sigma)_{\bar s}, \cO_{X(\Sigma)_{\bar s}}) = 0$ for $i>0$. Deformation theory shows that $I\to S$ is formally smooth at $\bar i$. We conclude that $I\to S$ is smooth, and in particular its image is an open subset of $S$.

We shall now prove that the assertion of (iii) holds for \emph{every} geometric point of $S$, or in other words, that $I\to S$ is surjective. By assumption, $I$ is non-empty, and hence the image of $I\to S$ is a non-empty open subset of $S$. It is also dense in $S$ (as $S$ is irreducible), so it only remains to show that it is closed under specialization. By Lemma~\ref{lemma:surjectivity-dvr}, it is enough to consider $S = \Spec V$ for a~discrete valuation ring $V$. In this case, Proposition~\ref{prop:toric-specialization} implies the required assertion. Now since $I\to S$ is smooth and surjective, it admits sections \'etale locally on $S$, which shows (ii).

Finally, assume (ii), and let $T={\rm Aut}^0_S(X, D)$ be the connected component of ${\rm Aut}_S(X, D)$. Then by Lemma~\ref{lemma:autom-xsigma}, \'etale-locally on $S$ there exists an isomorphism $T\isom \GG_m^n$ and an equivariant isomorphism $(X, D)\isom (X(\Sigma), D(\Sigma))$. Thus $(X, D)$ is a toric fibration over~$S$.
\end{proof}

\begin{lemma} \label{lemma:autom-xsigma}
  Let $\Sigma \subseteq \RR^n$ be a smooth complete fan, and let $S$ be a scheme. Then the group scheme $A={\rm Aut}_S((X(\Sigma), D(\Sigma))_S)$ is an extension 
  \[ 
    1\ra \GG_{m, S}^n\ra A\ra {\rm Aut}(\Sigma)_S\ra 1
  \]
  where ${\rm Aut}(\Sigma)$ is the (finite) group of automorphisms of the fan $\Sigma$.
\end{lemma}

\begin{proof}
Let $\Sigma(1)$ denote the set of rays (one-dimensional cones) of $\Sigma$. Let $D(\Sigma) = \bigcup_{\rho \in \Sigma(1)} D_{\rho}$ be a decomposition of $D(\Sigma)$ into irreducible components. Then every automorphism of the pair $(X(\Sigma), D(\Sigma))$ over $S$ has to permute the divisors $D_\rho$, which yields a homomorphism ${A\to {\rm Aut}(\Sigma(1))}$. Let $A^0$ denote its kernel and let $\sigma\in \Sigma$ be a top-dimensional cone, corresponding to an open immersion 
\[
  (\bb{A}^n_S, \{x_i=0\}_{i=1, \ldots, n})\hookrightarrow (X(\Sigma), D(\Sigma))_S
\]
whose image is preserved by $A^0$. A direct calculation shows that ${\rm Aut}_S(\bb{A}^n_S, \{x_i=0\}_{i=1, \ldots, n})$ can be naturally identified with $\GG_{m, S}^n$. 
\end{proof}

\begin{remark}
In the proof of Lemma~\ref{lemma:autom-xsigma}, the completeness assumption was only used to find a~top-dimensional cone in $\Sigma$. On the other hand, $\Aut_S(\GG_{m, S})$ (where $\GG_{m, S}$ is treated as an $S$-scheme) might be bigger than the semidirect product of $\GG_{m, S}$ and $\ZZ\times 2\ZZ$ (and non-representable) if $S$ is non-reduced. Smoothness is probably not necessary.
\end{remark}

\begin{lemma} \label{lemma:surjectivity-dvr}
  Let $S$ be a connected noetherian scheme and let $U\subseteq S$ be an open subset. Suppose that for every discrete valuation ring $V$, and for every morphism $h\colon \Spec V\to S$ mapping the generic point into $U$, we have $h(\Spec V)\subseteq U$. Then $U=S$.
\end{lemma}

\begin{proof}
The valuative criterion of properness shows that $U\to S$ is proper, and hence $U$ is also closed.
\end{proof}

\begin{remark}
Smoothness is probably not necessary for Proposition~\ref{prop:global-rigidity}. Properness seems essential: for example, the pair $(\bb{A}^2, \GG_m\times \{0\} + \{0\}\times \GG_m)$ has non-trivial first-order deformations.  Projectivity seems to be only an artifact of the proof.
\end{remark}

\subsection{Images of toric varieties}
\label{ss:images}

We shall now combine the technique of descending \mbox{$F$-lifta}\-bility (Theorem~\ref{thm:descending-frob-lift}) with the main result of \S\ref{ss:generalization} to show the following.

\begin{thm} \label{thm:c1impliesc2}
  Suppose that Conjecture~\ref{conj:froblift} is true for simply connected (e.g.\ separably rationally connected) varieties, that is, that every smooth projective simply connected $F$-liftable variety over an algebraically closed field of characteristic $p>0$ is a toric variety. Then Conjecture~\ref{conj:jarekw} is true, that is, a smooth projective image of a complete toric variety in characteristic zero is a toric variety.
\end{thm}

\begin{proof}
  Let $\phi_K\colon Z_K\to X_K$ be a surjective morphism from a complete toric variety $Z_K$ to a~smooth projective variety $X_K$ defined over an algebraically closed field $K$ of characteristic zero. Reasoning as in \cite{occhetta_wisniewski}, we can assume that $\phi_K$ is finite, in which case it is also flat by ``Miracle Flatness'' \cite[15.4.2]{EGAIV_3}, because $X_K$ is smooth and $Z_K$ is Cohen--Macaulay. There exists a finitely generated subring $R \subseteq K$ and a finite flat surjective map $\phi\colon Z\to X$ of schemes over $S=\Spec R$, satisfying the following properties:
  \begin{enumerate}[(1)]
    \item $\phi_K$ is the base change of $\phi$ to $K$,
    \item $S$ is smooth over $\Spec \ZZ$, 
    \item $Z$ is a proper constant toric fibration over $S$ (see \S\ref{ss:toric-var}),
    \item $X$ is smooth and projective over $S$,
    \item $d=\deg(\phi_K)$ is invertible on $S$.
  \end{enumerate}
  If $\bar s\colon \Spec k\to S$ is a geometric point of $S$, then $\phi_{\bar s}\colon Z_{\bar s}\to X_{\bar s}$ is a finite surjective map from a toric variety $Z_{\bar s}$ to a smooth projective variety $X_{\bar s}$ over $\bar s$, of degree invertible in $k$. Let $\swt s\colon \Spec W_2(k)\to S$ be a lifting of $\bar s$ mod $p^2$ (such a lifting exists thanks to condition (2)). Then $Z_{\swt s}$ is a constant toric fibration over $\swt s$, and hence $F_{Z_{\bar s}}$ lifts to $Z_{\swt s}$. By Theorem~\ref{thm:descending-frob-lift}(a) applied to $\phi_{\swt s}\colon Z_{\swt s}\to X_{\swt s}$, we get that $F_{X_{\bar s}}$ lifts to $X_{\swt s}$. Moreover, since $X_K$ is simply connected, so is $X_{\bar s}$, as the specialization map $\pi_1(X_K)\to \pi_1(X_{\bar s})$ is surjective \cite[Exp.~X, Corollaire~2.3]{SGA1}. The assumed case of Conjecture~\ref{conj:froblift} implies that $X_{\bar s}$ is a toric variety. Thus $X_K$ is toric, by Corollary~\ref{cor:toric-generalization}.
\end{proof}

\begin{remark}
If we believe in Conjecture~\ref{conj:froblift} in its full strength, then the connection between Conjecture~\ref{conj:jarekw} and a special case of Conjecture~\ref{conj:froblift} suggests that in characteristic zero, images of toric fibrations over abelian varieties should be also of this type, up to a finite \'etale cover. This would be a common generalization of Conjecture~\ref{conj:jarekw} and the results of \cite{Debarre,HwangMok01,DemaillyHwangPeternell} on the images of abelian varieties. 

The first obvious obstacle in deducing such a statement from Conjecture~\ref{conj:froblift} is the ordinary reduction conjecture for abelian varieties. But even assuming that, the method of proof of Theorem~\ref{thm:c1impliesc2} does not apply in this case, as the assumptions of Theorem~\ref{thm:descending-frob-lift}(a) may no longer be satisfied. More precisely, if $A\to S$ is an abelian scheme over a base $S$ which is of finite type over $\ZZ$, if $\bar s=\Spec k\to S$ is a closed geometric point, and $\swt s=\Spec W_2(k)\to S$ is a lifting of $\bar s$, then to apply the argument from the proof of Theorem~\ref{thm:c1impliesc2} we need not only $A_{\bar s}$ to be ordinary, but also $A_{\swt s}$ to be its Serre--Tate canonical lifting. We do not know whether one should expect such $\swt s$ to exist, even if $A$ is an elliptic curve over $S$.

Using Remark~\ref{remark:descend_reductive} we can show the following.  
\begin{quote} {\it
  Let $K$ be an algebraically closed field of characteristic zero and let $Z$ be a smooth projective variety over $K$ whose Albanese morphism $Z\to A$ is a toric fibration. Let $G$ be a finite group acting of $Z$. Suppose that $X=Z/G$ is smooth, and that the abelian variety $A$ satisfies the ordinary reduction conjecture. Assume that Conjecture~\ref{conj:froblift} is valid. Then $X$ admits a finite \'etale cover by a variety whose Albanese morphism is a~Zariski-locally trivial fibration with toric fibers.
} \end{quote}
Since the proof is long and technical, we refrained from including it in this article. 

Using Galois closures and Albanese mappings it is not difficult to show that if ${f \colon A \to X}$ is a finite morphism from an abelian variety $A$ to a smooth variety $X$ defined over an algebraically closed field $k$ of characteristic zero, then $X$ is a quotient of some (possibly different) abelian variety by a finite group. Therefore, the above result partially recovers the classification of smooth images of abelian varieties under finite morphisms mentioned above contingent upon the validity of Conjecture~\ref{conj:froblift}. 
\end{remark}


\section{\texorpdfstring{Structure of $F$-liftable nc pairs with trivial canonical bundle}{Structure of F-liftable nc pairs with trivial canonical bundle}}
\label{s:log-mehta-srinivas}

\subsection{Statement of the main result and some preliminaries}
\label{ss:log-mehta-srinivas-statement}

In this section, we provide a logarithmic generalization of \cite[Theorem 2]{MehtaSrinivas}, settling a special case of Conjecture~\ref{conj:froblift}.  More precisely, we prove the following theorem which characterizes projective $F$-liftable nc pairs $(X,D)$ with $\omega_{X}(D)$ numerically trivial. 

\begin{thm} \label{thm:frobenius_cy}
  Let $(X, D)$ be a projective nc pair over an algebraically closed field $k$ of positive characteristic. The following conditions are equivalent:
  \begin{enumerate}[(i)]
    \item $(X, D)$ is $F$-liftable and $\omega_{X}(D)$ is numerically trivial,
    \item $X$ is $F$-split and $\Omega^1_X(\log D)$ becomes trivial on a finite \'etale cover of $X$,
    \item $X$ admits a finite \'etale cover $f \colon Y \to X$ whose Albanese map $a\colon Y\to A$ is a toric fibration over an ordinary abelian variety with toric boundary $f^{-1}(D)$.
  \end{enumerate}
\end{thm}

\begin{proof}
We start by showing (i)$\Rightarrow$(ii). If $X$ is $F$-liftable, it is $F$-split, by Proposition~\ref{prop:frobenius_cotangent_morphism}(c). Using Variant~\ref{var:log_variant_xi}, we observe that there exists an injective morphism 
\[
  \xi \colon F^*\Omega^1_{X}(\log D) \ra \Omega^1_{X}(\log D).
\]  
The determinant of $\xi$ gives rise to a non-zero section of the $(p-1)$-st power of the numerically trivial bundle $\omega_X(D)$, and hence is an isomorphism.  Therefore, we see that $F^*\Omega^1_{X}(\log D) \isom \Omega^1_{X}(\log D)$ and thus by \cite[Satz 1.4]{Lange_Stuhler} the bundle $\Omega^1_{X}(\log D)$ becomes trivial on a finite \'etale cover $\pi \colon Y \to X$.

Now we show (iii)$\Rightarrow$(i).  Replacing $A$ with a finite \'etale cover and $Y$ with its base change, we may assume that $Y\to A$ is a split toric fibration.  By Lemma~\ref{lemma:flift-torsors}, there exists a~natural Frobenius lifting of $a\colon Y\to A$ over the Serre--Tate canonical lifting $(\wt A, \wt F_A)$ of $A$, and consequently $Y$ is $F$-liftable. Moreover, this Frobenius lifting is compatible with the toric boundary $f^{-1} D$ of $Y\to A$. Further, since $Y\to A$ is a split toric fibration with toric boundary $f^{-1}(D)$, the bundle $\smash{\Omega^1_{Y/A}(\log f^{-1} D)}$ is trivial, and hence $\omega_Y(f^{-1}(D))$ is trivial as well. Thus $(Y, f^{-1} D)$ satisfies (i), and hence by what we already proved it also satisfies (ii). Replacing $Y$ by a further finite \'etale cover, we can thus assume that $\Omega^1_Y(\log f^{-1} D)$ is trivial and that $Y\to X$ is Galois under a finite group $G$.

To finish the argument, we argue as in  \cite[Proof of Theorem~2, (i)$\Rightarrow$(ii)]{MehtaSrinivas}. By Variant~\ref{var:log_deformation_theory_frobenius}, Frobenius liftings of $(Y, f^{-1} D)$ correspond to splittings of the short exact sequence  
\[ 
  0 \ra B^1_{Y} \ra Z^1_{Y}(\log f^{-1} D) \ra \Omega^1_{Y'}(\log f^{-1} D') \ra 0,
\]
where $Y'$ is the Frobenius twist of $Y$. Since $\Omega^1_{Y}(\log f^{-1} D)$ is trivial and $H^0(Y, B^1_{Y})=0$ (as $Y$ is $F$-split), the above extension admits a unique splitting. In particular, this splitting is $G$-invariant, and hence descends to a splitting of the corresponding sequence on $X$. Thus $(X, D)$ is Frobenius liftable. Since $f^*\omega_X(D) = \omega_Y(f^{-1} D) \isom \cO_Y$, we see that $\omega_X(D)$ is numerically trivial.  

Finally, (ii)$\Rightarrow$(iii) follows from Theorem~\ref{thm:ordinary_trivial_cotangent}, whose proof occupies the subsequent subsection.
\end{proof}

The proof of the implication (ii)$\Rightarrow$(iii) in the case $D=0$ in \cite{MehtaSrinivas} relied on lifting to characteristic zero and using  Yau's work on the Calabi conjecture. Instead, we use the following two characteristic zero results. 

\begin{thm}[{\cite[Corollary 1]{Winkelmann}}]\label{thm:winkelmann}
  Let $(X,D)$ be a projective nc pair defined over $\bb{C}$.  Then the following conditions are equivalent:
  \begin{enumerate}[(i)]
    \item the log cotangent bundle $\Omega^1_X(\log D)$ is trivial,
    \item there exists a semi-abelian group variety $T$ (an extension of an abelian variety by a~torus) acting on $X$ with an open dense orbit $X \setminus D$.
  \end{enumerate}
  If $\dim \Alb X = 0$, then $(X,D)$ is a toric pair.
\end{thm}

We say that a group $G$ is \emph{virtually abelian} if it contains a finitely generated abelian subgroup of finite index.
\begin{thm}[{\cite[Theorem 10.1]{greb_kebekus_peternell}, cf.\ \cite{nakayama-zhang}}]\label{thm:greb_kebekus_peternell}
  Let $X$ be a normal projective variety over $\bb{C}$ which admits a polarized endomorphism of degree greater than one.  Then the topological fundamental group $\pi_1(X(\bb{C}))$ is virtually abelian.  The same holds for $\pi^{\et}_1(X)$ since it is the profinite completion of $\pi_1(X(\bb{C}))$.
\end{thm} 

\begin{remark} \label{rmk:false-theorems}
Since \cite{MehtaSrinivas} was motivated by the corresponding result in characteristic zero, we must warn the reader that the following statements are false.

\begin{quote}
  {\bf ``Theorem'' A }(characteristic zero analogue of Theorem~\ref{thm:frobenius_cy}){\bf.} \emph{Let $(X, D)$ be a~projective nc pair over an algebraically closed field $k$ of characteristic zero. The following conditions are equivalent:
  \begin{enumerate}[(i)] 
    \item $\Omega^1_X(\log D)$ becomes trivial on a finite \'etale cover of $X$,
    \item $X$ admits a finite \'etale cover $f \colon Y \to X$ whose Albanese map $a\colon Y\to A$ is a~toric fibration over an ordinary abelian variety with toric boundary $f^{-1}(D)$.
  \end{enumerate}}  

  \vspace{0.3cm}  
  \noindent {\bf ``Theorem'' B.} \emph{The assertion of Winkelmann's theorem (Theorem~\ref{thm:winkelmann}) holds over algebraically closed fields of positive characteristic.} 

  \vspace{0.3cm}
  \noindent {\bf ``Theorem'' C.} \emph{Let $X$ be a smooth projective scheme over an open subset $S \subseteq \Spec \cO_K$ where $\cO_K$ is the ring of integers in a number field $K$. Suppose that for infinitely many closed points $s\in S$, the reduction $X_{\bar s}$ has a finite \'etale cover whose Albanese map is a toric fibration. Then the same holds for $X_{\overline K}$.}
  \vspace{0.3cm}
\end{quote}

In ``Theorem'' A, we have (i)$\Rightarrow$(ii) by Theorem~\ref{thm:winkelmann}, but (ii) does \emph{not} imply (i). A basic counterexample is as follows. Let $C$ be an elliptic curve, let $L$ be a line bundle of non-zero degree on $C$, and let $E=\cO_C\oplus L$. Let $X=\PP_C(E)$ and let $D$ be the sum of the two sections $C\to X$ corresponding to the projections $E\to \cO_C$ and $E\to L$. Then condition (ii) is satisfied with $Y=X$. Moreover, every finite \'etale cover of $X$ is of the same type, so if (ii)$\Rightarrow$(i) were to hold, we could find such an $X$ with $\Omega^1_X(\log D)$ trivial. So suppose that $\Omega^1_X(\log D)$ is trivial. Again by Theorem~\ref{thm:winkelmann}, the open subset $U=X\setminus D$ (which is the $\GG_m$-torsor corresponding to $L$) admits a group structure making $U\to C$ a~group homomorphism. But this is only possible if $\deg L = 0$ by the ``Barsotti--Weil formula'' 
\[ 
  \Pic^0(C) \isom \Ext^1(C, \GG_m) \quad (\text{see \cite[VII.16, Th\'eor\`eme 6]{SerreGACC}}).
\]
Note that in the two extreme cases $A=Y$ and $A=0$, the implication (ii)$\Rightarrow$(i) does indeed hold. This construction also provides a counterexample to ``Theorem'' B (i)$\Rightarrow$(ii). 

For a counterexample to ``Theorem'' C, we take an elliptic curve $C$ over $\ZZ[1/N]$ for some $N$, and the non-split extension
\[ 
  0\ra \cO_C \ra E\ra \cO_C\ra 0.
\]
We again set $X=\PP_C(E)$. If $p$ is a prime of ordinary reduction of $E$ and $k=\overline{\FF}_p$, then $X_k$ satisfies the required property by Remark~\ref{rmk:partial-converse} and 
\cite[Section 3]{PartII}. On the other hand, the above extension does not become split over any finite \'etale cover of $C_{\overline\QQ}$, and consequently no finite \'etale cover of $X_{\overline\QQ}$ is a toric fibration over an abelian variety. 
\end{remark}

\subsection{\texorpdfstring{$F$-split nc pairs with trivial cotangent bundle}{F-split nc pairs with trivial cotangent bundle}}

In this section, we prove the following result generalizing \cite[Theorem 1]{MehtaSrinivas} and yielding the implication (ii)$\Rightarrow$(iii) needed above.

\begin{thm}[{Log version of \cite[Theorem 1]{MehtaSrinivas}}]\label{thm:ordinary_trivial_cotangent}
Let $(X,D)$ be a projective nc pair over an algebraically closed field $k$ such that $\Omega^1_{X}(\log D)$ is trivial and $X$ is $F$-split.  Then after an \'etale covering $X$ admits a structure of a toric fibration over an ordinary abelian variety.
\end{thm}

Even though the proof closely follows the ideas of Mehta and Srinivas, there are many important details that need to be figured out in the logarithmic setting.  For this reason, we precede the proof with a sequence of lemmas generalizing their results. 

\begin{lemma}[{Log version of \cite[Lemma 1.2]{MehtaSrinivas}}]\label{lem:logMS12}
  Let $(X,D)$ be an nc pair satisfying the hypotheses of Theorem~\ref{thm:ordinary_trivial_cotangent}, and let $\pi \colon Y \to X$ be a finite \'etale covering.  Then $(Y,\pi^{-1}D)$ also satisfies the hypotheses of Theorem~\ref{thm:ordinary_trivial_cotangent}.
\end{lemma}

\begin{proof}
Since $\pi$ is \'etale, we see that $\Omega^1_{Y}(\log \pi^{-1}D)$ is isomorphic to $\pi^*\Omega^1_{X}(\log D)$ and is therefore trivial.  To prove that $Y$ is $F$-split, we use Lemma~\ref{lem:f-split_is_ordinary}(d).
\end{proof}

\begin{lemma}[{Log version of \cite[Lemma 1.4]{MehtaSrinivas}}]\label{lem:logMS14}
  Let $(X,D)$ be an nc pair satisfying the hypotheses of Theorem~\ref{thm:ordinary_trivial_cotangent}.  Then the Albanese map $f \colon X \to \Alb X$ is smooth, surjective, and has connected fibers. Moreover, $(X, D)$ is an nc pair over $\Alb X$ and $\Omega^1_{X/\Alb X}(\log D)$ is trivial.
\end{lemma}

\begin{proof}
First, we observe that the logarithmic differential $f^*\Omega^1_{\Alb X} \to \Omega^1_{X}(\log D)$ decomposes as
\[
  f^*\Omega^1_{\Alb X} \xrightarrow{df} \Omega^1_{X}  \hookrightarrow \Omega^1_{X}(\log D).
\]
Since $H^0(\Alb X,\Omega^1_{\Alb X}) = H^0(X,f^*\Omega^1_{\Alb X})$, by \cite[Lemma 1.3]{MehtaSrinivas} we see that the induced morphism $H^0(X,f^*\Omega^1_{\Alb X}) \to H^0(X,\Omega^1_{X})$ is injective and hence 
\begin{align}
  H^0(X,f^*\Omega^1_{\Alb X}) \ra H^0(X,\Omega^1_{X}(\log D)) \label{eq:logMS14_item1}
\end{align} 
is injective as well.  As $f^*\Omega^1_{\Alb X}$ and $\Omega^1_{X}(\log D)$ are trivial, the map ${f^*\Omega^1_{\Alb X} \to \Omega^1_{X}(\log D)}$ is uniquely determined by \eqref{eq:logMS14_item1} and is therefore injective.  We conclude that $f$ is separable and dominant and that there exists a~short exact sequence of trivial bundles
\[
  0 \ra f^*\Omega^1_{\Alb X} \ra \Omega^1_{X}(\log D) \ra \Omega^1_{X/\Alb X}(\log D) \ra 0.
\]
Consequently, the sheaf of relative log differentials $\Omega^1_{X/\Alb X}(\log D)$ is locally free and therefore $f \colon (X,D) \to \Alb X$ is an nc pair over $\Alb X$ by Lemma~\ref{lemma:nc-pair-over-s-and-over-y}.  

To see that the fibers of $f$ are connected, that is, we have $f_*\cO_X = \cO_{\Alb X}$, we consider the Stein factorization $X \to Z \to \Alb X$.  Since $X \to \Alb X$ is smooth, the morphism $Z \to \Alb X$ is \'etale (by \cite[Exp.\ X, Prop.\ ~1.2]{SGA1} or \cite[7.8.10 (i)]{EGAIII_II}) and therefore $Z$ is an abelian variety.  By the universal property of the Albanese morphism, $Z \to \Alb X$ is an isomorphism, which finishes the proof.
\end{proof}

\begin{lemma}\label{lem:log_trivial_lifting_frobenius}
  Let $(X,D)$ be an nc pair satisfying the hypothesis of Theorem~\ref{thm:ordinary_trivial_cotangent}.  Then there exists a projective lifting $(\cX,\cD)$ of $(X, D)$ over $W(k)$ together with a lifting of the Frobenius morphism $F_{\cX} \colon (\cX,\cD) \to (\cX,\cD)$.  Moreover, for every line bundle $L$ on $X$ there exists a line bundle $\cL$ on $\cX$ such that $L \isom \cL_{|X}$ and $F_{\cX}^*\cL \isom \cL^{\otimes p}$.
\end{lemma}

\begin{proof}
We apply Variant~\ref{var:log_deformation_theory_frobenius} to see that the obstruction classes to lifting the nc pair $(X,D)$ together with the Frobenius morphism over consecutive Witt rings $W_n(k)$ lie in 
\[
  \Ext^1(\Omega^1_{X}(\log D),B^1_X).
\]  
Since $\Omega^1_{X}(\log D)$ is trivial and $X$ is $F$-split, we see by Lemma~\ref{lem:f-split_is_ordinary} that the latter group satisfies
\[
  \Ext^1(\Omega^1_{X}(\log D),B^1_X) = \Ext^1(\cO_X^{\oplus n},B^1_{X}) = H^1(X,B^1_X)^{\oplus n} =  0
\]
and therefore $(X,D)$ deforms to a formal nc pair $\{(X_n,D_n)\}_{n\geq 1}$ over the formal spectrum of $W(k)$ together with a compatible lifting of the Frobenius morphism 
\[
  \{F_{X_n} \colon (X_n,D_n) \to (X_n,D_n)\}_{n\geq 1}.
\]
Since $X$ is $F$-split, by Lemma~\ref{lem:f-split_is_ordinary} we see that the Frobenius action on $H^i(X,\cO_X)$ is bijective and hence by Proposition~\ref{prop:deformation_theory_frobenius}(d) we see that every line bundle $L$ on $X$ admits a formal lifting $\{L_n \in \Pic(X_n)\}_{n \geq 1}$ such that $F_{X_n}^*L_n \isom L_n^{\otimes p}$.  

To finish the proof we need to show that the given inductive systems are algebralizable.  For this purpose, since every ample line bundle deforms to the formal nc pair $\{(X_n,D_n)\}_{n\geq 1}$, we may just apply Grothendieck's algebraization theorem (see \cite[Section 3.4]{EGAIII_I} or \stacksproj{089A}).
\end{proof}

\begin{remark} \label{remark:log_trivial_lifting_frobenius_finite_field}
The lifting of the Frobenius morphism we exhibit above is not a $W(k)$-linear endomorphism, it is only Frobenius-linear.  However, if $X$ is defined over a finite field $k = \FF_{p^e}$, then the $e^{\rm th}$ iterate of $F_{\cX}$ is in fact a polarized $W(k)$-endomorphism of $\cX$.
\end{remark}

The following lemma is the essential part of our argument.  It differs substantially from its counterpart in \cite{MehtaSrinivas}, and is based on the two theorems stated in \S\ref{ss:log-mehta-srinivas-statement}. 

\begin{lemma} \label{lem:logMS_toric_fibers}
  Let $(X,D)$ be an nc pair satisfying the hypotheses of Theorem~\ref{thm:ordinary_trivial_cotangent}.  Then either $(X,D)$ is a toric pair (see \S\ref{ss:toric-var}) or there exists a finite \'etale covering $Y \to X$ such that $\Alb Y \neq 0$. 
\end{lemma}

\begin{proof} 
First, we assume that there exists an \'etale covering $Y \to X$ such that $\smash{H^1(Y,\cO_{Y}) = 0}$.  In this case, we claim that $(X,D)$ is in fact a toric pair.  In order to see this, we use Lemma~\ref{lem:log_trivial_lifting_frobenius} to obtain a $W(k)$-lifting $(\cY,\cE)$ of $(Y,E)$, where $E$ is the preimage of~$D$.  The induced  deformation of the log cotangent bundle $\Omega^1_{Y}(\log E)$ is trivial because the tangent space of its deformation functor is isomorphic to 
\[
  H^1\left(Y,\cEnd(\Omega^1_{Y}(\log E))\right) = H^1(Y,\cO_{Y})^{\oplus n^2} = 0.
\]
This implies that the log cotangent bundle of the generic fiber $\cY_{\eta}$ is trivial.  By  semicontinuity of cohomology we also see that $H^1(\cY_{\bar{\eta}},\cO_{\cY_{\bar{\eta}}}) = 0$ and therefore $(\cY_{\bar{\eta}},\cE_{\bar{\eta}})$ is a toric pair by Theorem~\ref{thm:winkelmann}.  By Proposition~\ref{prop:toric-specialization} this means that the special fiber $(Y,E)$ is a toric pair as well.  To finish the proof of the claim, we show that $Y \to X$ is an isomorphism.  For this purpose, we observe using the Hirzebruch--Riemann--Roch theorem (see \cite[Corollary 15.2.1]{fulton_intersection}) that 
\[
  1 = \chi(Y, \cO_{Y}) = d \cdot \chi(X, \cO_X),
\]
where $d$ is the degree of the finite map $Y \to X$.  This clearly implies that $d = 1$ and hence $Y \to X$ is an isomorphism.

Now we proceed to the second case, where $H^1(Y,\cO_{Y}) \neq 0$ for every \'etale covering $Y \to X$.  We follow the strategy of \cite[Lemma 1.6]{MehtaSrinivas} substituting the application of the Calabi conjecture with Theorem~\ref{thm:greb_kebekus_peternell} describing the algebraic fundamental groups of varieties admitting a polarized endomorphism.  We claim that there exists an \'etale covering $Y \to X$ such that $\Alb Y \neq 0$.  

For the proof, we first use the spreading out technique to reduce to the case of nc pairs defined over finite fields.  Let $(\cX,\cD)$ be an nc pair over a spectrum of a finitely generated local $\FF_p$-algebra $R$ such that the geometric generic fiber is isomorphic to $(X,D)$ and the residue field is finite with $q = p^e$ elements.  Spreading out the trivialization of the log cotangent bundle and the Frobenius splitting (which can interpreted as a morphism of vector bundles on the Frobenius twist of $X$), we may assume that $(\cX, \cD)_{\overline{\FF}_q}$ satisfies the  assumptions of Theorem~\ref{thm:ordinary_trivial_cotangent}.  Assume now that there exists a finite \'etale covering 
\[
  \pi_{\FF_q} \colon \cY_{\overline{\FF}_q} \ra \cX_{\overline{\FF}_q}
\]
of the geometric special fiber such that $\Alb \cY_{\overline{\FF}_q} \neq 0$.  By \cite[Arcata, IV, Proposition~2.2]{SGA4.5} (or \stacksproj{0BQC}) we see that such a covering extends to a covering of $\cX \times_{\Spec R} {\Spec R^{\rm sh}}$, where $R^{\rm sh}$ is the strict henselization of $R$.  Since $R^{\rm sh}$ is the colimit of \'etale extensions of $R$ we observe that, possibly after taking an \'etale covering of $\Spec R$, the morphism $\pi_{\FF_q}$ arises as a special fiber of a covering $\pi \colon \cY \to \cX$. The geometric generic fiber of $\pi$ yields an \'etale covering of $Y \to X$ which satisfies $\dim \Alb Y = H^1_{\et}(Y,\QQ_l) \neq 0$ by smooth base change for \'etale cohomology (we take $\ell \neq p$).  This finishes the reduction step.

For $(X,D)$ defined over a finite field $k = \FF_q$, we reason in two steps which we describe briefly.  First, we lift the pair $(X,D)$ to characteristic zero and apply Theorem~\ref{thm:greb_kebekus_peternell} to prove that $\pi_1^\et(X)$ is virtually abelian.  Consequently, using the approach described in \cite[Lemma 1.6]{MehtaSrinivas}, we prove that for an \'etale cover $Y \to X$ the crystalline cohomology group $H^1_{\rm crys}(Y/K)$ over the field $K = {\rm Frac}\,W(k)$ is non-zero, and hence $\dim \Alb Y = \dim_K H^1_{\rm crys}(Y/K)$ is non-zero as well.

Now, we present the details of the first step.  Using Lemma~\ref{lem:log_trivial_lifting_frobenius} we construct a $W(k)$-lifting $((\cX,\cD),F_{\cX})$ of $((X,D),F_X)$.  We set $\eta$ to be the generic point of $W(k)$.  Then, applying Remark~\ref{remark:log_trivial_lifting_frobenius_finite_field}, we observe that $F_X^e \colon X \to X$ is in fact a~$k$-linear endomorphism for some $e>0$, and therefore the geometric generic fiber $F^e_{\cX_{\bar{\eta}}}$ of the lifting of $F_X^e$ is a~polarized endomorphism of $\cX_{\bar{\eta}}$.  By Theorem~\ref{thm:greb_kebekus_peternell} the geometric fundamental group $\piet(\cX_{\bar{\eta}})$ is virtually abelian.  Using the surjectivity of the specialization morphism $\piet(\cX_{\bar{\eta}}) \to \piet(X)$ we see that the same holds for $\piet(X)$ and therefore after taking an \'etale covering of $X$ we may assume $\piet(X)$ is abelian.  

We proceed to the second step.  We assume that $\piet(X)$ is abelian and follow \cite[Lemma 1.6]{MehtaSrinivas} closely.  First, by \cite[Chapitre II, Th\'eor\`eme 5.2]{illusie_de_rham_witt}, to prove that $\dim \Alb(X) = \dim_K H^1_{\rm crys}(X/K)$ is non-zero it suffices to show that $H^1_{\et}(X,\ZZ_p)$ is non-torsion.  We have
\[
  H^1_{\et}(X,\ZZ_p) \simeq \Hom_{\ZZ}(\piet(X)^{\wedge p},\ZZ_p),
\] 
where $\piet(X)^{\wedge p} = \smash{\varprojlim_n} \piet(X) \otimes_{\ZZ} \ZZ/p^n$.  The $\ZZ_p$-module $\piet(X)^{\wedge p}$ is finitely generated and therefore it is torsion if and only of it is finite.  If $\piet(X)^{\wedge p}$ was finite, then there would exist an \'etale covering $X' \to X$ such that $H^1_{\et}(X',\FF_p) = 0$.  This gives a contradiction with $H^1(X',\cO_{X'}) \neq 0$.  Indeed, using the Artin--Schreier sequence of \'etale sheaves
\[
  0\ra \FF_p \ra \cO_X \stackrel{\scriptscriptstyle 1-F}\ra \cO_X \ra 0
\] 
we see that $H^1(X',\FF_p)$ is the $\FF_p$-vector space of elements in $H^1(X',\cO_{X'})$ fixed by the Frobenius morphism.  This is non-zero because Frobenius is bijective on $H^1(X',\cO_{X'})$ for an $F$-split scheme by Lemma~\ref{lem:f-split_is_ordinary}(a).
\end{proof}

\begin{lemma}[{Log version of \cite[Lemma 1.7]{MehtaSrinivas}}]\label{lem:logMS17}
  Let $(X,D)$ be an nc pair satisfying the hypothesis of Theorem~\ref{thm:ordinary_trivial_cotangent}, and let $Y \to X$ be a Galois \'etale cover with $\Alb Y \neq 0$.  Then there exists an intermediate Galois \'etale cover $Z \to X$ of degree $p^m$, for some $m \geq 0$, such that $Y \to Z$ induces an isogeny on Albanese varieties, in particular $\Alb Z \neq 0$.
\end{lemma}

\begin{proof}
We apply the argument given in the proof of \cite[Lemma 1.7]{MehtaSrinivas} to the logarithmic cotangent bundle instead of cotangent bundle.
\end{proof}

\begin{lemma}[{Log version of \cite[Lemma 1.8]{MehtaSrinivas}}]\label{lem:logMS18}
  Let $(X,D)$ be an nc pair satisfying the hypothesis of Theorem~\ref{thm:ordinary_trivial_cotangent}, and let $\pi \colon X \to \Alb X$ be the Albanese mapping.  Then 
  \begin{enumerate}[(a)]
    \item all geometric fibers of $\pi$ are $F$-split, 
    \item for each $i \geq 0$, $R^i\pi_* \cO_X$ is a locally free $\cO_{\Alb X}$-module which becomes free on a finite \'etale cover of $\Alb X$, 
    \item $\Alb X$ is an $F$-split abelian variety.
  \end{enumerate}
\end{lemma}
\begin{proof}
This follows by the same proof with the caveat that the Cartier isomorphism and Grothendieck duality need to be replaced with their logarithmic versions. For the convenience of the reader, we present a slightly simplified argument below.

For the proof of (a) it is sufficient to show that $\pi$ is relatively $F$-split, which follows from \cite[Theorem 1.2]{albanese_f_split}.  To get (b), we reason as in \cite[Lemma 1.8]{MehtaSrinivas}.  More precisely, let $A = \Alb X$, let $X'$ be the base change of $X$ along $F_A$, and let $\pi' \colon X' \to A$ denote the induced projection.  By flat base change we obtain an isomorphism ${F_A^*R^i\pi_*\cO_X \isomto R^i\pi'_*\cO_{X'}}$.  Since the fibers of $\pi'$ are $F$-split we observe that the natural homomorphisms ${R^i\pi'_*\cO_{X'} \to R^i\pi'_*F_{X/A*}\cO_{X}}$ induced from the long exact sequence of $R^i\pi'_*$ for
\[
  0 \ra \cO_{X'} \ra F_{X/A*}\cO_X \ra B^1_{X/A} \ra 0
\] 
are isomorphisms and therefore $F_A^*R^i\pi_*\cO_X \isom R^i\pi_*\cO_{X}$.  By \cite[Lemma 1.4]{mehta_nori} we see that $R^i\pi_*\cO_{X}$ is locally free, and \cite[Satz 1.4]{Lange_Stuhler} implies that it is \'etale trivializable. Part (c) follows from Lemma~\ref{lem:f-split_is_ordinary}(c). 
\end{proof}

\begin{lemma}[{Log version of \cite[Lemma 1.9]{MehtaSrinivas}}]\label{lem:logMS19}
  Let $(X,D)$ be an nc pair satisfying the hypothesis of Theorem~\ref{thm:ordinary_trivial_cotangent}, and let $\pi \colon X \to \Alb X$ be the Albanese mapping.  Let $s \in \Alb X$ be a closed point.  Then every finite \'etale covering of $X_s$ of degree $p^m$, for any $m\geq 0$, is induced by a covering of $X$.
\end{lemma}

\begin{proof}
Given Lemma~\ref{lem:logMS18}(b) the proof of \cite[Lemma 1.9]{MehtaSrinivas} applies without change.  
\end{proof}

\begin{lemma}[Generalization of {\cite[Remark below Lemma~1.10]{MehtaSrinivas}}]\label{lem:MS_remark}
  Let $Y$ be a smooth projective variety and let $\pi \colon X \to Y$ be a proper smooth morphism whose geometric fibers are $F$-split.  Then $R^1\pi_*\QQ_{\ell}$ becomes trivial on a finite \'etale cover of $Y$. 
\end{lemma}

\begin{proof}[Remark about the proof]
Once we know that the Albanese varieties of smooth projective \mbox{$F$-split} varieties are ordinary (see Lemma~\ref{lem:f-split_is_ordinary}), the proof given in {\cite[Remark below Lemma~1.10]{MehtaSrinivas}} can be repeated word for word.
\end{proof}

Equipped with the above we proceed to:

\begin{proof}[Proof of Theorem~\ref{thm:ordinary_trivial_cotangent}]
Let $(X,D)$ be a nc pair satisfying the hypotheses of Theorem~\ref{thm:ordinary_trivial_cotangent}.  We want to prove that there exists an \'etale covering admitting a structure of a toric  fibration over an ordinary abelian variety with a toric divisor given by the preimage of $D$.  First, following \cite[Proof of Theorem 1.1]{DemaillyHwangPeternell} we consider an \'etale covering $Y \to X$ such that $\dim \Alb Y$ is maximal (by Lemma~\ref{lem:logMS14} we know that $\dim \Alb Y \leq \dim Y = \dim X$).  Using Lemma~\ref{lem:logMS12} we see that for $E$ defined as the preimage of $D$ the nc pair $(Y,E)$ also satisfies the hypotheses of Theorem~\ref{thm:ordinary_trivial_cotangent}.  

We claim that the Albanese morphism of $Y$ is a toric fibration with toric boundary $E$. Indeed, using Lemma~\ref{lem:logMS14} and Lemma~\ref{lem:logMS18} we see that the fibers of $Y \to \Alb Y$ satisfy the assumptions of Lemma~\ref{lem:logMS_toric_fibers} and therefore they are either toric, and hence the proof is finished by Proposition~\ref{prop:global-rigidity}, or for some fiber $Y_s$ there exists an  \'etale covering $g_s \colon \bar{Y}_s \to Y_s$ such that $\Alb \bar{Y}_s \neq 0$.  In the latter case, using Lemma~\ref{lem:logMS17}, we may assume that $\deg g_s = p^m$, for some $m \geq 0$, and therefore by Lemma~\ref{lem:logMS19} the morphism $g_s$ is induced by a covering $\bar{Y} \to Y$.  We consider the composition $\pi \colon \bar{Y} \to \Alb Y$ of the covering and the Albanese morphism of~$Y$.  By Lemma~\ref{lem:logMS12} the fibers of $\pi$ are $F$-split and hence we may apply Lemma~\ref{lem:MS_remark} to see that $R^1\pi_*\QQ_{\ell}$ is non-zero and \'etale trivializable.  This means that after an \'etale covering $\eta \colon A \to \Alb Y$ we have an isomorphism
\[
  \eta^*R^1\pi_*\QQ_{\ell} \simeq \QQ_{\ell}^{2d}
\]
for some $d > 0$. Let $\widehat{Y} = \bar{Y} \times_{\Alb Y} A$ be the covering of $\bar{Y}$ induced by $\eta$, and let $\widehat{\pi} \colon \widehat{Y} \to A$ be the projection in the cartesian diagram
\[
\xymatrix{
  \widehat{Y} \ar[d]_{\widehat{\pi}} \ar[rr] & & \bar{Y} \ar[d]^{\pi} \\
  A \ar[rr]_{\eta} & & \Alb Y.
}
\]  
We claim that $\dim \Alb \widehat{Y} > \dim \Alb Y$.  Indeed, as in \cite[Proof of Theorem 1]{MehtaSrinivas} we consider the Leray spectral sequence $H^i_{\et}(A,R^j\widehat{\pi}_*\QQ_{\ell}) \Rightarrow H^{i+j}_{\et}(\widehat{Y},\QQ_{\ell})$ to obtain the exact sequence
\[
  0 \ra H^1_{\et}(A,\QQ_{\ell}) \ra H^1_{\et}(\widehat{Y},\QQ_{\ell}) \ra H^0(A,R^1\widehat{\pi}_*\QQ_{\ell}) \ra H^2_{\et}(A,\QQ_{\ell}) \ra H^2_{\et}(\widehat{Y},\QQ_{\ell}).
\]
Since $\widehat{\pi} \colon \widehat{Y} \to A$ admits a multi-section the morphism $H^2_{\et}(A,\QQ_{\ell}) \ra H^2_{\et}(\widehat{Y},\QQ_{\ell})$ is injective, and therefore we have
\begin{align*}
  \dim \Alb \widehat{Y} = \frac 1 2 \dim_{\QQ_{\ell}} H^1(\widehat{Y},\QQ_{\ell}) & = \frac 1 2 \left(\dim_{\QQ_{\ell}} H^1(A,\QQ_{\ell}) + \dim_{\QQ_{\ell}} H^0(A,R^1\widehat{\pi}_*\QQ_{\ell})\right) \\ 
  & = \dim A + \frac 1 2 \dim_{\QQ_{\ell}} H^0(A,\eta^*R^1\pi_*\QQ_{\ell}) \\
  & = \dim A + \frac 1 2 \dim_{\QQ_{\ell}} H^0(A,\QQ_{\ell}^{2d}) \\
  & = \dim \Alb Y + d > \dim \Alb Y.
\end{align*}
This gives a contradiction with the choice of $Y$ and hence finishes the proof. 
\end{proof}


\section{Homogeneous spaces}
\label{s:homogeneous-spaces}

In \cite[\S 4]{BTLM}, the authors study Frobenius liftability of rational homogeneous spaces. In many cases, they are able to show that such a variety $X$ is not $F$-liftable because Bott vanishing 
\[ 
    H^j(X, \Omega^i_X \otimes L) = 0 \quad (j> 0, \quad L\text{ ample})
\]
does not hold. As in this paper (although using a different argument), they reduce the question to the case of Picard number one. But even under this assumption, finding $i, j$, and $L$ as above for which $H^i(X, \Omega^j_X\otimes L)\neq 0$ is a difficult task. To this end, the authors use involved results of D.~Snow \cite{snow-hermitian,snow-grassmann} on the cohomology of flag varieties of Hermitian symmetric type in characteristic zero. They ask (see \cite[Remark~2]{BTLM}) whether the only $F$-liftable rational homogeneous spaces are products of projective spaces. As these are precisely the toric ones, this is a special case of our Conjecture~\ref{conj:froblift}, one which we are actually able to settle. 

Our proof in the Picard number one case does not rely on the classification of homogeneous spaces or Bott vanishing. In fact, we only need to assume that $X$ is Fano and that the tangent bundle $T_X$ is nef (note that the Campana--Peternell conjecture (\cite{campana-peternell}, \cite[V, Conjecture 3.10]{kollar96}) predicts that these two conditions should in fact imply that $X$ is a rational homogeneous space). The main ingredients of our proof are Mori's characterization of the projective space in terms of rational curves (Theorem~\ref{thm:mori}) and a careful analysis of the restrictions of the sheaf of $\xi$-invariant forms (\S\ref{subs:frob-cover}) to some special families of rational curves (Proposition~\ref{prop:homog-rank-one}). For the general case, we unfortunately have to look at the classification of homogeneous spaces, but only to check for which vertices of which Dynkin diagrams the corresponding homogeneous space of Picard rank one is a projective space (Lemma~\ref{lemma:pr-homogeneous}). A result of Lauritzen and Mehta \cite{LauritzenMehta} allows us to exclude the possibility of non-reduced stabilizers. We believe that our ideas could be useful in tackling the Picard rank one case of Conjecture~\ref{conj:froblift} with the assumption that $T_X$ is nef dropped. 

In this section we work over an algebraically closed field $k$ of positive characteristic.

\subsection{Families of rational curves} 
\label{subs:rational_curves} 

We start by recalling basic definitions pertaining to rational curves. The main reference for this subsection is \cite[II \S 2--3]{kollar96}. In what follows we assume that $X$ is a smooth projective $k$-scheme.

\begin{defin} \label{def:free-curve}
  Let $\phi \colon \PP^1 \to X$ be a non-constant morphism. We say that $\phi$ is \emph{free} if $\phi^*T_X$ is nef, and \emph{very free} if $\phi^*T_X$ is ample.
\end{defin}

Given a rational curve $C \subseteq X$ we shall say that it is free (resp.\ very free) if the normalization $\phi \colon \PP^1 \to C \to X$ is free (resp.\ very free). Further, for a free $\phi \colon \PP^1 \to X$ we can write 
\[ 
  \phi^* T_X \isom \cO_{\PP^1}^{\oplus r} \oplus \bigoplus_{i=1}^{n-r} \cO_{\PP^1}(a_i), 
  \quad
  a_i > 0.
\] 
With that, $\phi$ is very free if and only if $r=0$.

Fix an ample divisor $H$ on $X$. If $X$ is Fano, we always take $H = -K_X$. For an integer $d$, we denote by $\mathrm{Hom}_{d}(\PP^1, X) \subseteq {\rm Hilb}(\PP^1 \times X)$ the scheme parametrizing morphisms $\phi \colon \PP^1 \to X$ such that $\deg \phi^* H = d$. We denote by $\mathrm{Hom}^{\rm free}_{d}(\PP^1, X) \subseteq \mathrm{Hom}_{d}(\PP^1, X)$ the subscheme parametrizing free $\phi \colon \PP^1 \to X$ which are generically injective. We drop the subscript $d$ whenever we do not wish to specify the degree, so $\mathrm{Hom}(\PP^1, X) = \bigsqcup_d \mathrm{Hom}_{d}(\PP^1, X)$. By \cite[II, Corollary 3.5.4]{kollar96}, the natural morphism 
\[
  \PP^1 \times \Hom^{\rm free}_{d}(\PP^1,X) \ra X
\]
is smooth, and so $\Hom^{\rm free}_{d}(\PP^1,X)$ is smooth as well. 

Let $\mathrm{RatCurves}^{\rm free}_d(X)$ be the quotient $\Hom^{\rm free}_{d}(\PP^1,X) / \Aut(\PP^1)$, which exists by \cite[II, Comment 2.7]{kollar96} as $\Hom^{\rm free}_{d}(\PP^1,X)$ is a smooth open subscheme of $\mathrm{Hom}_d(\PP^1,X)$, invariant under the $\Aut(\PP^1)$-action (see \cite[II, Corollary 3.5.4]{kollar96}). We denote by $\mathrm{Univ}_d(X)$ the universal $\PP^1$-bundle over $\mathrm{RatCurves}^{\rm free}_d(X)$ so that we have a diagram
\[
  \xymatrix{
    \mathrm{Univ}_d(X) \ar[r]^-{\phi} \ar[d]_{\pi} & X \\
    \mathrm{RatCurves}^{\rm free}_d(X) & 
  }
\]
It is constructed as a quotient of $\PP^1 \times \Hom^{\rm free}_d(\PP^1, X)$ by $\Aut(\PP^1)$.  We have a factorization 
\[
  \PP^1 \times \Hom^{\rm free}_{d}(\PP^1,X) \ra \mathrm{Univ}_d(X) \overset{\phi}\longrightarrow X
\]
where the left arrow is smooth and surjective (cf.\ \cite[II, Theorem 2.15]{kollar96}), and so $\phi$ is smooth. In particular, both $\mathrm{Univ}_d(X)$ and $\mathrm{RatCurves}^{\rm free}_d(X)$ are smooth \stacksproj{02K5}.

\begin{defin}
  Let $\phi \colon \PP^1 \to X$ be a generically injective morphism which is free and such that  $d = \deg \phi^*H$ is minimal among all $\phi' \in \mathrm{Hom}^{\rm free}(\PP^1, X)$. We call such $\phi$ a \emph{minimal free rational curve}.
\end{defin}

The study of deformations of rational curves of minimal degree plays a vital role in the theory of rationally connected varieties by means of the Mori theory (and \emph{varieties of minimal rational tangents}, see e.g.\ \cite{kebekus02,hwang-mok04}). The picture becomes slightly simpler if we assume that $T_X$ is nef. As we have already noted, conjecturally this condition should be equivalent to $X$ being a homogeneous space. When $T_X$ is nef, we denote $\mathrm{RatCurves}^{\rm free}_d(X)$ by $\mathrm{RatCurves}_d(X)$, as all rational curves are free.

\begin{lemma} [{\cite[II, Proposition 2.14.1]{kollar96}}] \label{lem:deformations_of_curves_tx_nef} 
  With notation as above, suppose that $T_X$ is nef and  that minimal free rational curves are of degree $d$. Then $\mathrm{RatCurves}_d(X)$ is proper.
\end{lemma}

\begin{proof}
By \cite[II, Theorem 2.15]{kollar96}, the above definition of $\mathrm{RatCurves}_d(X)$ coincides with \cite[II, Definition -- Proposition 2.11]{kollar96} under the assumption that $T_X$ is nef, and hence $\mathrm{RatCurves}_d(X)$ is proper by \cite[II, Proposition 2.14.1]{kollar96}.
\end{proof}

Finally, let us recall the celebrated result of Mori, which we present here in a slightly different from than that of {\cite[V, Theorem 3.2]{kollar96}} -- for the proof of Proposition \ref{prop:homog-rank-one} we need a variant for minimal free rational curves (cf.\ \cite[Families of curves, Theorem]{occhetta_wisniewski}). 

\begin{thm} \label{thm:mori}
  Let $X$ be a smooth projective Fano variety of dimension $n$ defined over an algebraically closed field $k$ and let $x \in X$ be a general point. Suppose that every rational curve through $x$ is free, and that every minimal free rational curve through $x$ is very free. Then $X \isom \PP^n$.
\end{thm}

Every rational curve through $x$ is free provided that $k$ is of characteristic zero and $x$ is very general \cite[II, Theorem 3.11]{kollar96}, or $X$ is $F$-liftable and $x \in X$ is general (Lemma~\ref{lem:main_lemma_of_magic_covers}(a)). Note that there exists a minimal free rational curve through a general point $x$ by \cite[II, Corollary 3.5.4.2]{kollar96}.

\begin{proof} 
A very free rational curve must be of degree at least $n+1$ (see \cite[V, Lemma~3.7.2]{kollar96}). Moreover, there exists a rational curve of degree at most $n+1$ through $x$ by bend-and-break (see \cite[V, Theorem 1.6.1]{kollar96}), hence minimal free rational curves are of degree $n+1$. Since they are very free by assumption, the rest of the proof follows {\cite[V, Theorem 3.2]{kollar96}} word for word.   
\end{proof}

\subsection{\texorpdfstring{The sheaf of $\xi$-invariant forms and rational curves}{The sheaf of xi-invariant forms and rational curves}}
\label{subs:frob-cover}

In this subsection we study positivity conditions imposed on the tangent bundle by the existence of a Frobenius lifting. This is the main component of the proof of Theorem~\ref{thm:homog} in the Picard rank one case.

Let us fix a Frobenius lifting $(\wt{X}, \wt{F}_X)$ of a smooth $k$-scheme $X$ and consider the induced morphism  $\xi \colon F^* \Omega^1_X \to \Omega^1_X$. Recall that it is injective, and hence generically an isomorphism, by Proposition~\ref{prop:frobenius_cotangent_morphism}. Let $U \subseteq X$ to be the maximal open subset where $\xi$ is an isomorphism.

The following simple lemma allows for the study of families of rational curves by using the sheaf of $\xi$-invariant forms denoted by $(\Omega^1_X)^{\xi}$.

\begin{lemma} \label{lem:main_lemma_of_magic_covers} 
  Let $\phi \colon \PP^1 \to X$ be a non-constant morphism such that $\phi(\PP^1)\cap U \neq \emptyset$. Then,
  \begin{enumerate}[(a)]
    \item $\phi$ is free (cf.\ {\cite[Proposition 5]{xin16}}), 
    \item  $\phi^*(\Omega^1_X)^{\xi}$ contains a locally constant $\bb{F}_p$-subsheaf of rank $r = h^0(\PP^1, \phi^*\Omega^1_X)$.
  \end{enumerate}
\end{lemma}

\begin{proof}
 Since $\phi^*\Omega^1_X$ is locally free and $\phi^*(\xi) \colon F^*\phi^*\Omega^1_X \to \phi^*\Omega^1_X$ is generically an isomorphism, $\phi^*(\xi)$ must be an injection. 

Write $\phi^*\Omega^1_X = \bigoplus_{i=1}^n \cO_{\PP^1}(a_i)$, where $a_1 \geq a_2 \geq \cdots \geq a_n$. Since $\phi^*(\xi)$ is an injection, it induces a~non-zero morphism $\cO_{\PP^1}(pa_1) \to \cO_{\PP^1}(a_j)$ for some $1 \leq j \leq n$. This is only possible if $a_1 \leq 0$, and so $\phi^*T_X$ is nef, concluding (a). 

The morphism $\phi^*(\xi)$ induces a Frobenius-linear automorphism of the space of global sections $H^0(\PP^1, \phi^*\Omega^1_X)$, and so by taking $\phi^*(\xi)$-fixed points
\[
\cG_{\phi} = H^0(\PP^1, \phi^*\Omega^1_X)^{\phi^*(\xi)} \otimes_{\FF_p} (\FF_p)_{\PP^1}
\]
we get a constant $\bb{F}_p$-subsheaf of $\phi^*(\Omega^1_X)^{\xi}$ of rank $r$, where $(\FF_p)_{\PP^1}$ is the constant sheaf with value $\FF_p$ on $\PP^1$. 
\end{proof}

As mentioned before, for any free rational curve $\phi \colon \PP^1 \to X$ we can write $\phi^* \Omega^1_X = \bigoplus_{i=1}^{n-r} \cO_{\PP^1}(a_i) \oplus \cO_{\PP^1}^{\oplus r}$, where $a_i < 0$. In general, the value of $r$ is upper-semicontinuous under deformations of $\phi$. Here, we show that $r$ is invariant under deformations provided that $X$ is $F$-liftable and $\phi(\PP^1)$ intersects $U$.

\begin{prop} \label{prop:invariance_of_r} 
  With notation as above, let $\phi_i\colon \PP^1 \to X$ for $i \in \{1,2\}$ be two rational curves intersecting $U \subseteq X$ and lying in the same irreducible component $\cM \subseteq \mathrm{Hom}(\PP^1, X)$. Then $h^0(\PP^1, \phi_1^*\Omega^1_X) = h^0(\PP^1, \phi_2^*\Omega^1_X)$.
\end{prop}

\begin{proof}
We have the following diagram
\[
  \xymatrix{
    \cM \times \PP^1 \ar[r]^-{\phi} \ar[d]_{\pi} & X \\
    \cM &
  }
\]
Replacing $\cM$ by an open subset, we can assume that $\phi(\{m\} \times \PP^1)$ intersects $U$ for every closed point $m \in \cM$. Pick any closed point $m \in \cM$. Then, as in Lemma~\ref{lem:main_lemma_of_magic_covers} we have
\[
  h^0(\PP^1, \phi^*\Omega^1_X|_{\pi^{-1}(m)}) = \dim_{\FF_p} H^0(\PP^1, \phi^*\Omega^1_X|_{\pi^{-1}(m)})^{\xi} = \dim_{\FF_p}\, (\pi_*\phi^*(\Omega^1_X)^{\xi})_m,
\]
which is lower semi-continuous with respect to $m$ by Lemma~\ref{lem:pushforward_of_etale_sheaf}. The last equality holds by the proper base change theorem. On the other hand, $h^0(\PP^1, \phi^*\Omega^1_X|_{\pi^{-1}(m)})$ is upper semi-continuous with respect to $m$ by the semi-continuity theorem (\cite[III, Theorem 12.8]{hartshorne77}). Hence, it is constant over~$\cM$.
\end{proof}

We needed the following corollary of the proper base change theorem in the above proof.

\begin{lemma}\label{lem:pushforward_of_etale_sheaf} 
  Let $f \colon X \to S$ be a proper morphism of schemes, let $g \colon U \to X$ be a separated \'etale morphism of finite type, and let $\cF$ be the \'etale sheaf of sections of $g$. Then the function $\phi \colon S \to \ZZ$ defined as
  \[
    \phi(s) = |(f_*\cF)_{\bar s}|
  \]
  for any geometric point $\bar s$ over $s$ is lower semi-continuous. 
\end{lemma}

\begin{proof}
By \cite[Expos\'e IX, Corollaire 2.7.1]{SGA4III} or \stacksproj{03S8}, $\cF$ is constructible, and hence $f_*\cF$ is constructible as well \cite[Expos\'e XIV]{SGA4III}. Therefore $\phi$ is a constructible function on $S$, and hence to prove the assertion it is enough to show that if $s,\mu\in S$ are two points such that $s$ lies in the closure of $\mu$, then $\phi(s)\leq \phi(\mu)$. 

Let $\bar s$ be a geometric point lying over $s$ and let $\overline\mu$ be a geometric point of the localization $S_{(\bar s)}$ lying over $\mu$. Passing to stalks, we obtain the cospecialization map 
\[
  c_{\overline{\mu}\leadsto\bar{s}}\colon (f_*\cF)_{\bar s} \ra (f_*\cF)_{\overline{\mu}}.
\] 
It remains to show that $c_{\overline{\mu}\leadsto\bar s}$ is injective. The proper base change theorem implies that $(f_*\cF)_{\bar s}$ is the set of sections of $U \to X$ over $X \times_{S} S_{(\bar s)}$. Take any two such sections ${u_1, u_2 \colon X \times_{S} S_{(\bar s)} \to U}$, and suppose they are equal after restricting to $X \times_{S} S_{(\overline{\mu})} \to U$. Then $u_1=u_2$, as $U \to X$ is separated and $S_{(\overline{\mu})}  \to S_{(\bar s)}$ has dense image. 
\end{proof}

\subsection{The Picard rank one case}
In this subsection we prove Theorem~\ref{thm:homog} in the Picard rank one case (Proposition \ref{prop:homog-rank-one}).  Before proceeding with the proof we need the following result.

\begin{lemma} \label{lem:no_global_forms} 
  Let $X$ be a smooth projective $F$-liftable Fano variety over $k$. Then  $X$ is simply connected and $H^0(X, \Omega^1_X)=0$.
\end{lemma}

\begin{proof}
In order to prove that $X$ is simply connected, let us consider an \'etale cover ${f \colon Y \to X}$ of degree $d$. By the Hirzebruch--Riemann--Roch theorem (see \cite[Corollary 15.2.1]{fulton_intersection}), we have $\chi(Y, \cO_Y) = d\cdot \chi(X, \cO_X)$. Since $X$ is $F$-liftable, so is $Y$ (Lemma~\ref{lemma:ascending-frob-lift}), and hence Kodaira vanishing holds on both $X$ and $Y$ (see Theorem~\ref{thm:bott-vanishing}). Thus $\chi(Y, \cO_Y) = \chi(X, \cO_X) = 1$, which shows that $d=1$ and $X$ is simply connected.

Now, we show that $H^0(X, \Omega^1_X)=0$. Since $X$ is $F$-liftable, all of its global one-forms are closed by Proposition~\ref{prop:frobenius_cotangent_morphism}, and
\[
  Z^1_X \isom B^1_X \oplus \Omega^1_X,
\]
where $Z^1_X = \ker (d \colon F_* \Omega^1_X \to F_* \Omega_X^2)$ and $B^1_X = F_*\cO_X / \cO_X$. In particular, the Cartier operator induces an isomorphism $H^0(X, Z^1_X) \isom H^0(X,\Omega^1_X)$. Therefore, the assumptions of \cite[Proposition 4.3]{geer-katsura03} are satisfied and we obtain
\[
  H^0(X,\Omega^1_X) \isom \Pic(X)[p] \otimes_{\ZZ} k.
\]
We claim that $\Pic(X)$ is torsion free. Indeed, if $L$ is a numerically trivial line bundle on $X$, then by Kodaira vanishing we have
\[ 
  H^i(X,L) = H^i(X, \omega_X\otimes \omega_X^{-1} \otimes L) = 0 \quad \text{ for }i>0,
\]
and therefore $h^0(X,L) = \chi(X,L) = \chi(X, \cO_X) = h^0(X, \cO_X) = 1$.
\end{proof}

\begin{prop} \label{prop:homog-rank-one}
  Let $X$ be a smooth projective Fano variety of dimension $n$ defined over an algebraically closed field of positive characteristic. Suppose that $T_X$ is nef, $\rho(X)=1$, and $X$ is $F$-liftable. Then $X \isom \PP^n$.
\end{prop}

The idea of the proof is the following.  If $X \not \isom \PP^n$, then by Mori theory there exists a~rational curve $C \subseteq X$ such that $T_X|_C$ is not very ample. In particular, $h^0(C, \Omega^1_X|_C) > 0$.  Using the global lifting of Frobenius we show that the sections of $\Omega^1_X|_{C_t}$, where $C_t$ are deformations of $C$, glue to global sections of $\Omega^1_X$. This contradicts the fact that $X$ is Fano (see Lemma \ref{lem:no_global_forms}). 

\begin{proof}
Choose a general point $x \in X$. If every minimal free rational curve $\rho \colon \PP^1 \to X$ passing through $x$ is very free, then $X \isom \PP^n$ by Theorem~\ref{thm:mori}. Suppose by contradiction that there exists such a non-very-free minimal $\rho$ of degree $d$. Then $h^0(\PP^1, \rho^*\Omega^1_X) = r \geq 1$.

Let $\cM$ be the irreducible component of $\mathrm{RatCurves}_d(X)$ containing $\rho$ (see \S\ref{subs:rational_curves} for the notation). By Lemma~\ref{lem:deformations_of_curves_tx_nef}, we know that $\mathscr{M}$ is proper and there exist morphisms
\[
  \xymatrix{
    \mathrm{Univ} \ar[d]_{\pi} \ar[r]^{\phi} &X\\
    \cM & 
  }
\]
with $\pi \colon \mathrm{Univ} \to \cM$ being a $\PP^1$-fibration, and $\phi$ being smooth. Since $X$ is simply connected (see Lemma~\ref{lem:no_global_forms}), we get that $\phi$ has connected fibers. Indeed, let $\mathrm{Univ} \to Y \to X$ be the Stein factorization of $\phi$. Since $\phi$ is smooth and proper, the morphism $Y \to X$ is finite \'etale (see \cite[Exp.\ X, Prop.\ 1.2]{SGA1} or \cite[7.8.10 (i)]{EGAIII_II}). 
and so it must be an isomorphism.

We fix a Frobenius lifting $(\wt{X}, \wt{F}_X)$ of $X$ and consider the induced morphism  ${\xi \colon F^* \Omega^1_X \to \Omega^1_X}$. Let $(\Omega^1_X)^{\xi}$ be the sheaf of $\xi$-invariant forms and let $U \subseteq X$ be the maximal open subset for which $\xi|_U \colon F^* \Omega^1_U \to \Omega^1_U$ is an isomorphism. 

Set $\cM^\circ =  \pi(\phi^{-1}(U))$ to be the locus of all $m \colon \PP^1 \to X$ in $\cM$ whose image intersects $U$. It is open as $\pi$ is smooth. Take $\mathrm{Univ}^\circ = \pi^{-1}(\cM^\circ)$ and let $j\colon \mathrm{Univ}^\circ\to\mathrm{Univ}$ denote the inclusion. We define the following subsheaf of $\phi^{*}(\Omega^1_X)^{\xi}|_{\mathrm{Univ}^\circ}$
\[
  \mathscr{G} = j^* \pi^{*} \pi_* \phi^{*} (\Omega^1_X)^{\xi}|_{\mathrm{Univ}^\circ}.
\]
Given $m \colon \PP^1 \to X$ in $\cM^\circ$, the proper base change theorem implies that $\smash{\mathscr{G}|_{\pi^{-1}(m)}}$ is isomorphic to the constant $\FF_p$-sheaf with value $\smash{H^0(\PP^1, m^*(\Omega^1_X)^{\xi})}$. Lemma~\ref{lem:main_lemma_of_magic_covers}(b) and Proposition~\ref{prop:invariance_of_r} show that $\mathscr{G}$ is a locally constant $\FF_p$-sheaf of rank $r$. Indeed, $(\pi_*\phi^*\Omega^1_X)|_{\mathscr{M}^\circ}$ is a~locally free sheaf of rank $r$ by \cite[III, Corollary 12.9]{hartshorne77} and
\[
  \mathscr{G} = \pi^{*}(\pi_* \phi^* \Omega^1_X|_{\mathscr{M}^\circ})^{\pi_* \phi^*\xi}.
\]

We claim that $\mathrm{codim}\, \mathrm{Univ} \, \backslash\, \mathrm{Univ}^\circ \geq 2$. In order to prove this it is enough to show that for an irreducible divisor $D \subseteq X \backslash U$, the codimension of $\phi^{-1}(D) \backslash \mathrm{Univ}^\circ$ in $\mathrm{Univ}$ is at least two. Since $\phi$ is smooth and has connected fibers, $\phi^{-1}(D)$ is irreducible, and hence it is enough to show that $\phi^{-1}(D) \cap \mathrm{Univ}^\circ \neq \emptyset$. This follows from the fact that $\rho(X)=1$. Indeed, for a general closed point $m \in \cM$, the curve $\phi(\pi^{-1}(m))$ intersects $D$, and so $\pi^{-1}(m) \subseteq \mathrm{Univ}^\circ$ intersects $\phi^{-1}(D)$.

By Zariski--Nagata purity, the sheaf $j_*\cG$ is a locally constant $\FF_p$-sheaf of rank $r$. Since $\mathrm{Univ}$ is normal and the complement of $\mathrm{Univ}^\circ$ is of codimension at least two, we have $j_* j^* \phi^*\Omega^1_X \isom \phi^*\Omega^1_X$. In particular,
\[
  j_*\cG \subseteq j_* j^* \phi^*(\Omega^1_X)^{\xi} = (j_* j^* \phi^*\Omega^1_X)^{\phi^* \xi}  \isom \phi^{*}(\Omega^1_X)^{\xi}. 
\] 
Let $T$ be a fiber of $\phi \colon \mathrm{Univ} \to X$. Since $j_*\cG|_T$ is a locally constant subsheaf of the constant sheaf $\phi^*(\Omega^1_X)^{\xi}|_T$, we see that $j_*\cG|_T$ is constant. In particular, the proper base change theorem implies that $\phi_* j_*\cG$ is a  nonzero locally constant subsheaf of $(\Omega^1_X)^{\xi}$. As $X$  is simply connected, $\phi_* j_* \cG$ is constant, and hence $H^0(X, (\Omega^1_X)^{\xi}) \neq 0$, contradicting Lemma~\ref{lem:no_global_forms}.
\end{proof}

\subsection{The general case}

We recall some facts about homogeneous spaces over an algebraically closed field $k$. A projective variety $X$ over $k$ is called \emph{homogeneous} if its automorphism group $\Aut(X)$ acts transitively on $X$. Note that since $X$ is projective, $\Aut(X)$ is actually (the set of $k$-points of) a group scheme locally of finite type over $k$, and in particular its connected component $\Aut^0(X)$ is a group scheme of finite type acting transitively on $X$. 

The decomposition theorem of Borel and Remmert \cite{BorelRemmert} (see \cite[\S 5--6]{SanchoDeSalas} for its extension to positive characteristic) states that $X$ decomposes into a product
\begin{equation} \label{eqn:borel-remmert-decomposition}
  X\isom A\times G_1/P_1\times \ldots \times G_r/P_r,
\end{equation}
where $A$ is an abelian variety, $G_i$ are simple linear algebraic groups of adjoint type, and $P_i\subseteq G_i$ are parabolic subgroup schemes (i.e., each $P_i$ contains a Borel subgroup of $G_i$).

Let $G$ be a simple linear algebraic group over $k$. A choice of a Borel subgroup $B\subseteq G$ and a~maximal torus $T\subseteq B$ gives a set $D$ of \emph{simple roots} of $G$, which is the set of nodes of the Dynkin diagram of $G$. Following the conventions of \cite[p.\ 58]{humphreys}, we number them $1, \ldots, n$ where $n$ is the rank of $G$. Reduced parabolic subgroups of $G$ containing $B$ are in an~inclusion-preserving bijection with subsets of $D$. We can thus represent (possibly ambiguously) rational homogeneous spaces with reduced stabilizers by marked Dynkin diagrams, i.e.\ Dynkin diagrams with a chosen set of nodes. If $\alpha\in D$ is a simple root, we denote by $P(\alpha)$ the maximal parabolic subgroup corresponding to $D\setminus\{\alpha\}$.  

\begin{lemma} \label{lemma:pr-homogeneous}
  Let $G$ be as above, and let $\alpha\in D = \{1, \ldots, n\}$ be a simple root. Suppose that $G/P(\alpha)\isom \PP^r_k$ for some $r>0$. Then one of the following holds:
  \begin{enumerate}[(i)]
    \item $r=n$, the group $G$ is of type $A_{n}$, and $\alpha=1$ or $n$,
    \item $r=2n-1\geq 3$, the group $G$ is of type $C_n$, and $\alpha=1$.
  \end{enumerate}
\end{lemma}

In other words, the Dynkin diagram and the node $\alpha$ are as shown below.

\begin{center}
  \includegraphics[width=0.3\textwidth]{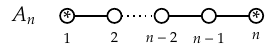} 
  \hspace{1cm}
  \includegraphics[width=0.3\textwidth]{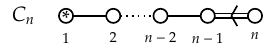} 
\end{center}

\begin{proof}
Demazure \cite{Demazure} showed that for most pairs $(G, P)$ of a simple group $G$ of adjoint type and a reduced parabolic subgroup $P$, the natural morphism
\[ 
  G \ra \Aut^0(G/P)
\]
is an isomorphism, and classified the exceptions. The only exception with $\Aut^0(G/P)\isom PGL_{r+1}(k)$ is (see \emph{op.cit.}, case (a) on p.\ 181) $G=PSp_{r}(k)$ and $P$ the stabilizer of a line, in which case $G/P\isom \PP^r_k$ ($r$ has to be odd in this case). This is our case (ii). If (ii) does not hold, then $r=n$ and $G\isom PGL_{n+1}(k)$, and hence $G/P(\alpha)\isom {\rm Gr}(\alpha, n+1)$. So $\dim G/P(\alpha) = \alpha(n+1-\alpha) = n$, and hence $\alpha=1$ or $\alpha=n$.
\end{proof}

Below,  we denote by $F_{1, n}$ the incidence variety parametrizing partial flags $W_1\subseteq W_{n}\subseteq k^{n+1}$ where $\dim W_1 = 1$ and $\dim W_{n} = n$. It is a hypersurface of degree $(1, 1)$ in $\PP^n_k\times \PP^n_k$, and as a homogeneous space it corresponds to the Dynkin diagram of type $A_{n}$ with all nodes except for the two endpoints $1, n$ marked. 

\begin{lemma} \label{lemma:pr-homogeneous-max}
  Let $G$ be a simple algebraic group over $k$ and let $P\subseteq G$ be a reduced parabolic subgroup. Suppose that for every maximal reduced parabolic $Q\subseteq G$ containing $P$, the homogeneous variety $G/Q$ is isomorphic to  $\PP^r_k$ for some $r$. Then one of the following conditions holds:
  \begin{enumerate}[(i)]
    \item $G/P\isom \PP^n_k$ for some $n$,
    \item $G/P\isom F_{1, n}$ for some $n$.
  \end{enumerate}
\end{lemma}

\begin{proof}
This follows from Lemma~\ref{lemma:pr-homogeneous} and the classification of rational homogeneous spaces with reduced stabilizers.
\end{proof}

\begin{lemma} \label{lemma:incidence-variety}
  The incidence variety $F_{1, n}$ is not $F$-liftable for $n>1$.
\end{lemma}

\begin{proof}
This is shown in {\cite[\S 4.2]{BTLM}} by using Bott vanishing. We give an alternative, more geometric proof in \cite[Section 2]{PartII}.
\end{proof}

\begin{thm}[\cite{LauritzenMehta}] \label{thm:reduced-stab}
  Let $G$ be a simple algebraic group over $k$, let $P\subseteq G$ be a parabolic subgroup scheme, let $X=G/P$, and let $Y=G/P_{\rm red}$. If $X$ is Frobenius split, then $X\isom Y$ as varieties.
\end{thm}

We are ready to show the main theorem of this section.

\begin{thm} \label{thm:homogeneous-spaces}
  Let $X$ be a projective variety over $k$ whose automorphism group acts transitively on $X(k)$. Then the following are equivalent:
  \begin{enumerate}[(i)]
    \item $X$ is $F$-liftable,
    \item $X\isom A\times \PP^{n_1}_k \times \ldots \times \PP^{n_r}_k$ for some $n_1, \ldots, n_r$, where $A$ is an ordinary abelian variety,
    \item the Albanese variety $A={\rm Alb}(X)$ is ordinary and the fibers of the Albanese map ${a_X\colon X\to A}$ are toric varieties.
  \end{enumerate}
\end{thm}

\begin{proof}
Consider the decomposition \eqref{eqn:borel-remmert-decomposition}. By Corollary~\ref{cor:flift-products}, $X$ is $F$-liftable if and only if $A$ and $G_1/P_1, \ldots, G_r/P_r$ are.  Moreover, $A$ is $F$-liftable if and only if it is ordinary, by Example~\ref{ex:ordinaryav}. The Albanese map of $X$ is simply the projection to the first factor of the decomposition. Since the only homogeneous toric varieties are products of projective spaces, we see that (ii) and (iii) are equivalent and imply (i).

It remains to prove that (i) implies (ii). For this, we can assume that $X=G/P$ where $G$ is a~simple linear algebraic group of adjoint type and $P\subseteq G$ is a parabolic subgroup scheme. Since $X$ is Frobenius split (Proposition~\ref{prop:frobenius_cotangent_morphism}(c)), Theorem~\ref{thm:reduced-stab} implies that $X\isom G/P_{\rm red}$, so we can assume that $P$ is reduced.

If $Q\subseteq G$ is a reduced maximal parabolic subgroup containing $P$, then $Z=G/Q$ inherits~a Frobenius lifting from $X$ via the map $X=G/P\to G/Q=Z$ by Theorem~\ref{thm:descending-frob-lift}(b), as the fibers are isomorphic to the rational homogeneous space $Q/P$ and $H^i(Q/P, \cO_{Q/P})=0$ for $i>0$. Now $Z$ is Fano, $\Pic Z\isom \ZZ$, and $T_Z$ is nef. By Proposition~\ref{prop:homog-rank-one}, $Z\isom \PP^n_k$ for some $n$. By Lemma~\ref{lemma:pr-homogeneous-max}, this implies that $X\isom \PP^n_k$ or $X$ is isomorphic to the incidence variety $F_{1, n}$. But $F_{1, n}$ is not $F$-liftable (Lemma~\ref{lemma:incidence-variety}), so $X\isom \PP^n_k$.
\end{proof}

\begin{remark}
By Theorem~\ref{thm:homogeneous-spaces} and the proof of Theorem~\ref{thm:c1impliesc2},  we get that Conjecture~\ref{conj:jarekw} holds if the target variety is a homogeneous space. This observation is not very interesting because this fact almost follows from \cite{occhetta_wisniewski}. What is missing is a direct proof that the incidence variety $F_{1, n}$ is not an image of a toric variety. 
\end{remark}


\bibliographystyle{amsalpha} 
\bibliography{bib.bib}

\renewcommand{\refname}{\rule{2cm}{0.4pt}}   
\renewcommand{\addcontentsline}[3]{}


\end{document}